\documentclass[12pt,leqno]{amsart}
\usepackage{amssymb,amsmath,amsthm,bm}
\usepackage{comment}
\usepackage{graphicx}
\usepackage{color}
\usepackage[textsize=tiny]{todonotes}
\usepackage[normalem]{ulem}
\usepackage[bookmarksopen,bookmarksdepth=3,colorlinks,citecolor=red,pagebackref,hypertexnames=false]{hyperref}
\usepackage[msc-links,nobysame,non-sorted-cites, initials]{amsrefs}
\usepackage[inline]{enumitem}
\usepackage{enumitem}
\setlist{topsep=0pt,leftmargin=*}
\usepackage{mathtools}
\usepackage{lipsum}
\usepackage{graphicx}
\usepackage{cancel}
\usepackage{url}
\usepackage{ulem}

\mathtoolsset{showonlyrefs}
\usetikzlibrary{quotes,arrows.meta}

\definecolor{darkblue}{RGB}{0,0,160}

\definecolor{ffqqqq}{rgb}{1,0,0}
\definecolor{qqqqff}{rgb}{0,0,1}
\definecolor{ffxfqq}{rgb}{1,0.5,0}
 
\usepackage{wrapfig}
\usepackage{pgfplots}
\usepgfplotslibrary{colormaps}
\usepackage{tikz-3dplot}
\usepackage{pgf,tikz}
\usetikzlibrary{matrix,arrows}

\usepackage[hyphenbreaks]{breakurl}
  
\usepackage[lining]{libertine}   
\usepackage{cabin}
\usepackage[libertine]{newtxmath}
 
\usepackage[T1]{fontenc}

\makeatletter
\newcommand*\bigcdot{\mathpalette\bigcdot@{.5}}
\newcommand*\bigcdot@[2]{\mathbin{\vcenter{\hbox{\scalebox{#2}{$\m@th#1\bullet$}}}}}
\makeatother

\def\eps{\varepsilon}

\def\d{{\mathrm d}}
\def\dist{{\mathrm {dist}}}

\def\supp{{\mathrm{supp}\, }}

\def\R{\mathbb{R}}
\def\N{\mathbb{N}}

\def\M {{\mathrm M}}

\def\Z {{\mathbb Z}}
\def\1 {{\mbox{\boldmath 1}}}

\def \l {\langle}

\def \r {\rangle}

\def\ind{\cic{1}}
\newcommand{\cic}{\bm}

\DeclareFontFamily{U}{mathx}{}
\DeclareFontShape{U}{mathx}{m}{n}{<-> mathx10}{}
\DeclareSymbolFont{mathx}{U}{mathx}{m}{n}
\DeclareMathAccent{\widecheck}{0}{mathx}{"71}

   \def\XXint#1#2#3{{\setbox0=\hbox{$#1{#2#3}{\int}$}
       \vcenter{\hbox{$#2#3$}}\kern-.5\wd0}}

\def \no#1#2#3 {{\bf #1} (#3), #2.}
\def \eds#1#2#3 {#1, #2, #3.}

\newcounter{counter}

\numberwithin{equation}{section}
\newtheorem{theorem}[counter]{Theorem}
\newtheorem*{theorem*}{Theorem}
\newtheorem{corollary}{Corollary}
\newtheorem{cor}[subsection]{Corollary}
\newtheorem{lemma}{Lemma} 
\newtheorem{proposition}[subsection]{Proposition}
\theoremstyle{definition}
\newtheorem{definition}{Definition}
\newtheorem*{remark*}{Remark}
\newtheorem{remark}[subsection]{Remark} 
\theoremstyle{plain}

\numberwithin{corollary}{counter}
\numberwithin{lemma}{section}

\DefineSimpleKey{bib}{myurl}

\newcommand\myurl[1]{\url{#1}}

\BibSpec{webpage}{%
  +{}{\PrintAuthors} {author}
  +{,}{ \textit} {title}
  +{}{ \parenthesize} {date}
  +{,}{ \myurl} {myurl}
  +{,}{ } {note}
  +{.}{ } {transition}
}
\author[O. Bakas]{Odysseas Bakas}
\address[O. Bakas]{Department of Mathematics, University of Patras, 26504 Patras, Greece}
\email{\href{mailto:obakas@upatras.gr}{\textnormal{obakas@upatras.gr}}}

\author[V. Ciccone]{Valentina Ciccone}
\address[V. Ciccone]{Hausdorff Center for Mathematics, Universit\"at Bonn, Endenicher Allee 60, 53115 Bonn, Germany }
\email{\href{mailto:ciccone@math.uni-bonn.de}{\textnormal{ciccone@math.uni-bonn.de}}}

\author[I. Parissis]{Ioannis Parissis}
\address[I.\ Parissis]{Departamento de Matem\'aticas, Universidad del Pa\'is Vasco, Aptdo. 644, 48080 Bilbao, Spain and Ikerbasque, Basque Foundation for Science, Bilbao, Spain}
\email{\href{mailto:ioannis.parissis@ehu.es}{\textnormal{ioannis.parissis@ehu.es}}}

\author[M. Vitturi]{Marco Vitturi}
\address[M. Vitturi]{Munster Technological University, Department of Mathematics, Bishopstown, Cork, Ireland}
\email{\href{mailto:marco.vitturi@ucc.ie}{\textnormal{marco.vitturi@ucc.ie}}}

\numberwithin{figure}{section}

\begin{document}

\thanks{O. Bakas was partially supported by the projects CEX2021-001142-S, RYC2018-025477-I, grant PID2021-122156NB-I00 funded by MICIU/AEI/10.13039/501100011033 and FEDER, UE, Juan de la Cierva Incorporaci\'on IJC2020-043082-I, grant BERC 2022-2025 of the Basque Government, and by the funding programme ``MEDICUS'' of the University of Patras.}

\thanks{This project was partially carried out during V. Ciccone's research visits to BCAM - Basque Center for Applied Mathematics and to the University of the Basque Country UPV/EHU. Her visits have been partially supported by the Hausdorff Center for Mathematics in Bonn through the Global Math Exchange Program and BIGS, and by the grant PID2021-122156NB-I00 funded by MICIU/AEI/10.13039/501100011033 and FEDER, UE, and grant IT1615-22 of the Basque Government.}

\thanks{I. Parissis is partially supported by grant PID2021-122156NB-I00 funded by MICIU/AEI/10.13039/501100011033 and FEDER, UE, grant IT1615-22 of the Basque Government and IKERBASQUE}

\subjclass[2010]{Primary: 42A45, 42A55, 42B25. Secondary: 42B35}
\keywords{Littlewood-Paley square functions, Marcinkiewicz multipliers, H\"ormander-Mihlin multipliers, endpoints bounds, lacunary sets of finite order}

\title[Higher order Marcinkiewicz multipliers]{Endpoint estimates for higher order Marcinkiewicz multipliers}

\begin{abstract}
We consider Marcinkiewicz multipliers of any lacunary order defined by means of uniformly bounded variation on each lacunary Littlewood--Paley interval of some fixed order $\tau\geq 1$. We prove the optimal endpoint bounds for such multipliers as a corollary of a more general endpoint estimate for a class of multipliers introduced by Coifman, Rubio de Francia and Semmes and further studied by Tao and Wright. Our methods also yield the best possible endpoint mapping property for higher order H\"ormander-Mihlin multipliers, namely multipliers which are singular on every point of a lacunary set of order $\tau$. These results can be considered as endpoint versions of corresponding results of Sj\"ogren and Sj\"olin. Finally our methods generalize a weak square function characterization of the space $L\log^{1/2}L$ in terms of a square function introduced by Tao and Wright: we realize such a weak characterization as the dual of the Chang--Wilson--Wolff inequality, thus giving corresponding weak square function characterizations for the spaces $L\log^{\tau/2}L$ for general integer orders $\tau\geq 1$.
\end{abstract}

\maketitle


\section{Introduction} Our topic is endpoint estimates for Marcinkiewicz-type multipliers on the real line. We recall that a Marcinkiewicz multiplier is a bounded function $m:\R\to \mathbb C$ which has bounded variation on each Littlewood--Paley interval $L_k\coloneqq (-2^{k+1},-2^k]\cup[2^k,2^{k+1})$, uniformly in $k\in\Z$. It is well known that the operator $\mathrm{T}_m f\coloneqq (m \hat f)^\vee$ is bounded on $L^p(\R)$ for all $p\in (1,\infty)$. Endpoint estimates for Marcinkiewicz multipliers were proved by Tao and Wright in \cite{TW} where the authors prove that they locally map $L\log^{1/2}L$ into weak $L^1$. 

A prototypical Marcinkiewicz multiplier is given by the signed sum
\[
\sum_{k\in \Z} \eps_k \ind_{L_k},\qquad \eps_k\in\{-1,+1\},
\]
while an orthogonality argument provides the link between Marcinkiewicz multipliers and the classical Littlewood--Paley square function
\[
\mathrm{LP}_1f(x)\coloneqq \left(\sum_{k\in\Z} |\mathrm P_kf|^2\right)^{\frac12}=\left(\mathbf E\left|\sum_{k\in\Z}\eps_k \mathrm{P}_k f\right|^2\right)^{\frac 12},\qquad \mathrm{P}_k f\coloneqq (\ind_{L_k}\hat f)^\vee,
\]
the expectation being over choices of independent random signs.

In the present paper we are interested in higher order versions of Marcinkiewicz multipliers. In order to motivate such a study it is very natural to consider square functions that project to Littlewood--Paley intervals given by lacunary sets of order $2$ or higher. For example letting 
\[
L_{(k,m)}\coloneqq \left \{\xi\in\R:\, |\xi|\in (2^k+2^{m-1},2^{k}+2^{m}]\cup[2^{k+1}-2^m,2^{k+1}-2^{m-1})\right\},\qquad k>m,
\]
denote the family of Littlewood--Paley intervals of second order, we naturally define
\[
\mathrm {LP}_2f \coloneqq \left(\sum_{\substack{(k,m)\in\Z^2\\k>m}} |\mathrm P_{(k,m)}f|^2\right)^{\frac12}=\left(\mathbf E\left|\sum_{\substack{(k,m)\in\Z^2\\k>m}}\eps_{k,m} \mathrm P_{(k,m)} f\right|^2\right)^{\frac 12},\quad \mathrm P_{(k,m)} f\coloneqq (\ind_{L_{(k,m)}}\hat f)^\vee,
\]
initially for Schwartz functions with compactly supported Fourier transform. This is a second order Littlewood--Paley square function while the multiplier 
\[
\sum_{\substack{(k,m)\in\Z^2\\k>m}} \eps_{(k,m)} \ind_{L_{(k,m)}},\qquad \eps_{(k,m)}\in\{-1,+1\},
\]
can be considered as a prototypical Marcinkiewicz multiplier of order $2$. A Littlewood--Paley partition $\{L: \, L\in\Lambda_\tau\}$ of lacunary order $\tau>1$ is naturally produced by iterating Whitney decompositions inside each Littlewood--Paley interval of order $\tau-1$. Accordingly, a Marcinkiewicz multiplier of order $\tau$ is a bounded function which has bounded variation uniformly on all Littlewood--Paley intervals of order $\tau$. Likewise, the Littlewood--Paley square function of order $\tau$ is
\[
\mathrm {LP}_\tau f\coloneqq \left(\sum_{L\in\Lambda_\tau } |\mathrm P_{L}f|^2\right)^{1/2}=\left(\mathbf E\left|\sum_{L\in\Lambda_\tau}\eps_{L}\mathrm P_{L}f\right|^2\right)^{1/2},\qquad \mathrm P_{L}f \coloneqq \left(\ind_{L}\hat f\right)^\vee.
\]
With precise definitions to follow, a punchline result of this paper is the following.

\begin{theorem}\label{thm:Marcink}\label{cor:LPsf}
 If $m$ is a Marcinkiewicz multiplier of order $\tau\in\N$ then $\mathrm{T}_m$ satisfies the estimate
\[
\left| \left\{ x\in\mathbb{R}: \, |\mathrm{T}_mf(x)|> \alpha \right\}\right|\lesssim \int_\mathbb{R} \frac{|f|}{\alpha} \bigg( \log \bigg(e+\frac{|f|}{\alpha} \bigg) \bigg)^{\tau/2}, \quad \alpha>0.
\]
The same is true for the Littlewood--Paley square function $\mathrm{LP}_\tau$ of order $\tau$. In both cases the endpoint estimates are best possible in the sense that the exponent $\tau/2$  in the right hand side cannot be replaced by any smaller exponent.
 \end{theorem}
We will deduce Theorem~\ref{thm:Marcink} as a consequence of the more general Theorem~\ref{thm:main} below which applies to the wider class of $R_{2,\tau}$ multipliers. 

\subsection{Lacunary sets of higher order}\label{sec:lacdef} In order to describe the classes of higher order multipliers we are interested in, it will be necessary to introduce some notation for lacunary sets of general order. The standard Littlewood--Paley partition of the real line is the collection of intervals $\Lambda_1\coloneqq \{ \pm [2^k, 2^{k+1}):\, k\in \Z\}$ and it is a Whitney decomposition of $\R\setminus \{0\}$. For a finite dyadic interval $I \subset \R$ the standard Whitney partition $\mathcal W(I)$ of $I$ is the collection of the maximal dyadic subintervals $L\subset I$ such that $\mathrm{dist}(L,\R\setminus I)=|L|$. Now for any integer $\tau>1$ we set
\[
\Lambda _\tau \coloneqq \bigcup_{I \in \Lambda_{\tau-1}}\mathcal{W}(I)
\]
and call $\Lambda_\tau$ the standard Littlewood--Paley collection of intervals of order $\tau$. We denote by $\mathrm{lac}_\tau$ the collection of all endpoints of intervals in $\Lambda_\tau$. Observe that, as in \cite{Bonami}, the set $\mathrm{lac}_\tau$ has the explicit representation
\[
\mathrm{lac}_\tau = \left\{ \pm 2^{n_1}\pm 2^{n_2}+\cdots\pm 2^{n_\tau}:\, n_1>n_2>\cdots>n_\tau,\, n_j\in\Z\quad\forall j\right\}.
\]

For uniformity in the notation we also set $\Lambda_0=
\{(-\infty,0),(0,+\infty)\}$ and $\mathrm{lac}_0\coloneqq\{0\}$. It will be useful throughout the paper to truncate the scales of lacunary intervals and numbers by defining
\[
\Lambda_\tau ^{n}\coloneqq\left\{L\in\Lambda_\tau:\, |L|\geq n\right\},\qquad n\in 2^\Z.
\]
Accordingly $\mathrm{lac}_\tau ^n$ denotes endpoints of intervals in $\Lambda_\tau ^n$.

We need a smooth way to project to frequency intervals in $\Lambda_\tau$. For this we consider a smooth even function $0\leq \eta \leq 1$ such that $\eta$ is identically $1$ on $[-1/2,1/2]$ and vanishes off $[-5/8,5/8]$. For a positive integer $
\tau$ and $L\in\Lambda_\tau$ we define the (rescaled) $L$-th frequency component of some multiplier $m:\R\to \mathbb C$ as
\[
m_{L}(\xi)\coloneqq \eta(\xi) m(c_L + \xi |L|),\qquad \xi\in\R,
\]
with $c_L$ denoting the center of $L$.
\subsection{Higher order multipliers and endpoint estimates}\label{sec:hodef} With this notation at hand we will say that $m:\R\to\mathbb C$ is a H\"ormander-Mihlin multiplier of order $\tau$ if 
\[
\|m\|_{H_\tau}\coloneqq \sup_{|\alpha|\leq M}\sup_{L\in\Lambda_\tau} \| \partial^\alpha m_{ L}\|_{L^\infty}<+\infty,
\]
for some sufficiently large positive integer $M$ which we will not keep track of. Note that the higher order H\"ormander-Mihlin condition is essentially the natural assertion
\[
|\partial^\alpha m(\xi)| \lesssim \dist(\xi,\mathrm{lac}_{\tau-1})^{-\alpha},\qquad \xi \in\R\setminus \mathrm{lac_{\tau-1}}.
\]
Likewise we will say that a bounded function $m:\R\to \mathbb C$ is a Marcinkiewicz multiplier of order $\tau\in\N$ if the components $m_{ L}$ have bounded variation uniformly in $L\in\Lambda_\tau$. Here we use the standard variation norms defined for $r\in[1,\infty]$ as follows
\[
\left\| F\right\|_{\mathrm{V}_r}\coloneqq \sup_N\sup_{\substack{x_0<\cdots<x_N}} \left(\sum_{0\leq k \leq N}|F(x_{k+1})-F(x_k)|^r \right)^{\frac1r}.
\]
Note that usually Marcinkiewicz multipliers are defined by asking that the pieces $m\ind_L$ have bounded $1$-variation, uniformly in $L$. One can check that our definition, using the smooth cutoff $\eta$, is equivalent to the classical one. For one inequality of this equivalence we just use that $\eta\equiv 1$ on $[-1/2,1/2]$, while for the converse inequality it suffices to notice that $\|FG\|_{\mathrm{V}_1}\lesssim \|F\|_{\mathrm{V}_1}\|G\|_{\mathrm{V}_1}$ together with the fact that the support of $\eta$ is contained in three adjacent intervals of length $1$. We will actually consider the wider class of $R_{2,\tau}$-multipliers defined below.

\begin{definition}Let $\mathcal R $ to be the space of all functions of the form
\[
m=\sum_{I}c_I \ind_I
\]
with $I$ ranging over a family of disjoint arbitrary subintervals in $[1,2)$ and the coefficients $\{c_I\}_I$ satisfying
\[
\sum_{I}|c_I|^2 \leq 1. 
 \]
Then $\overline {\mathcal R}$ is the Banach space of functions $m\coloneqq \sum_a \lambda_a m_a$ with $\sum_a |\lambda_a|<+\infty$; we equip $\overline{\mathcal R}$ with the norm
\[
\|m\|_{\overline {\mathcal R}} \coloneqq \inf\left\{\sum_a |\lambda_a|:\, m=\sum_a \lambda_a m_a,\quad m_a\in  \mathcal R  \right\}.
\]
For $\tau\in\N$ we say that the bounded function $m:\R\to\mathbb C$ is an $R_{2,\tau}$-multiplier if 
\[
\|m\|_{ R_{2,\tau}}\coloneqq \sup_{L\in\Lambda_\tau}\| m_{L}\|_{\overline{\mathcal R}}<+\infty.
\]
\end{definition}
The class $R_{2,\tau}$ contains all Marcinkiewicz multipliers of order $\tau$ as well as H\"ormander-Mihlin multipliers of order $\tau$. This follows by the fact that H\"ormander multipliers of order $\tau\geq 1$ are Marcinkiewicz multipliers of the same order and the latter belong to the class $\mathcal V_{1,\tau}$ consisting of functions which have uniformly bounded $1$-variation on each lacunary interval of order $\tau$; the inclusion relationship then follows for example by the fact that $\mathcal V_{1,\tau}\subset R_{2,\tau}$, proved in \cite[{Lemma 2}]{CRS}. Our main result proves the sharp endpoint bound for multipliers in the class $R_{2,\tau}$.
\begin{theorem}\label{thm:main}
Let $\tau$ be a positive integer and $m\in R_{2,\tau}$. Then the operator $\mathrm{T}_mf\coloneqq (m\hat f)^\vee$ satisfies
\[
\left| \left\{ x\in\R: \, |\mathrm{T}_mf(x)|> \alpha \right\}\right|\lesssim \int_{\mathbb{R}} \frac{|f|}{\alpha} \bigg( \log \bigg(e+\frac{|f|}{\alpha}\bigg) \bigg)^{\tau/2}, \quad \alpha>0.
\]
Furthermore this estimate is best possible in the sense that the exponent $\tau/2$ in the right hand side of the estimate cannot be replaced by any smaller exponent. The implicit constant depends only on $\tau$ and the $R_{2,\tau}$-norm of $m$.
\end{theorem}
For $\tau=1$ the local version of the theorem above is contained in \cite{TW}. We note that Theorem~\ref{thm:main} easily implies the following local estimate: For every interval $I$ and $m\in R_{2,\tau}$ there holds
\[
\left| \left\{ x\in I: \, |\mathrm{T}_mf(x)|> \alpha \right\}\right|\lesssim \frac{1}{\alpha} \int_I |f|  \bigg( \log \bigg(e+\frac{|f|}{\langle |f| \rangle_I}\bigg)  \bigg)^{\tau/2}, \qquad \alpha>0,\qquad \mathrm{supp} f\subset I ,
\] 
where $\langle |f| \rangle_I:=\vert I \vert^{-1}\Vert f \Vert_{L^1(I)}$.
The global estimate of Theorem~\ref{thm:main} appears to be new even in the first order case $\tau=1$, although a proof of a global result can be deduced for the first order case $\tau=1$ from the methods in \cite{TW} without much additional work.

While H\"ormander-Mihlin multipliers are $R_{2,\tau}$ multipliers, they are in general much better-behaved as the case $\tau=1$ suggests: indeed for $\tau=1$ H\"ormander-Mihlin multipliers map $L^1$ to $L^{1,\infty}$, in contrast to the sharpness of the $L\log^{1/2}L\to L^{1,\infty}$ estimate for general Marcinkiewicz or $R_{2,1}$ multipliers. In analogy to the Littlewood--Paley square function $\mathrm{LP}_\tau$ of order $\tau$ it is natural to define a smooth version as follows. For $C>0$, $M\in\N$ and $L\in \Lambda_\tau$ we consider the class of bump functions
\[
\Phi_{L,M}\coloneqq \left\{\phi_L :\, \supp(\phi_L)\subseteq \frac54L,\quad \sup_{\alpha\leq M}|L|^\alpha \|\partial^\alpha \phi_L\|_{L^\infty}\leq 10^{10}\right\}.
\]
Now for some fixed large positive integer $M$ (whose precise value is inconsequential) suppose that  $\phi_L \in \Phi_{L,M}$ for all $L\in\Lambda_\tau$ and define, initially for $f$ in the Schwartz class,
\[
\mathrm S_\tau f\coloneqq \left(\sum_{L\in\Lambda_\tau} |\Delta_{L} f|^2\right)^{1/2},\qquad \Delta_{L} f(x)\coloneqq \int_{\R}  \phi_{L} (\xi)\widehat f(\xi) e^{2\pi i x \xi}\,\d \xi,\qquad x\in \R.
\]
The following theorem is the sharp endpoint estimate for higher order H\"ormander-Mihlin multipliers and corresponding square functions.

\begin{theorem}\label{thm:horm} Let $\tau$ be a positive integer and $m\in H_\tau$ be a H\"ormander-Mihlin multiplier of order $\tau$. Then
\[
\left| \left\{ x\in\mathbb{R}: \, |\mathrm{T}_mf(x)|> \alpha \right\}\right|\lesssim \int_\mathbb{R} \frac{|f|}{\alpha} \bigg( \log \bigg(e+\frac{|f|}{\alpha}\bigg) \bigg)^{(\tau-1)/2}, \qquad \alpha>0.
\]
The same holds for the smooth Littlewood--Paley square function $\mathrm S_\tau$ of order $\tau$ and these results are best possible. The implicit constant depends only on $\tau$ and the $H_\tau$-norm of $m$, and also on $M$ in the case of square functions.
\end{theorem}
The case $\tau=1$ of this corollary is classical. The local version of the case of H\"ormander-Mihlin multipliers of order $\tau=2$ is implicit in \cite{TW} as it can be proved by combining \cite[{Proposition 5.1}]{TW} with \cite[{Proposition 4.1}]{TW}. All the higher order cases for the multipliers of the class $H_\tau$ appear to be new.

\subsection{The Chang--Wilson--Wolff inequality and a square function for \texorpdfstring{$L\log^{\tau/2}L$}{Llog L}} We work on the probability space $([0,1],\d x)$ throughout this section  unless otherwise stated. A central result in the approach in \cite{TW} was a weak characterization of the space $L\log^{1/2}L$ in terms of an integrable square function, inspired by the analogous and better-known characterisation of the Hardy space $H^1$. More precisely, the authors in \cite{TW} prove that if $f\in L\log^{1/2}L$ and $f$ has mean zero then for each $L\in\Lambda_1$ one can construct nonnegative functions $F_L$ such that
\begin{equation}\label{eq:TWsf}
|\Delta_L f| \lesssim F_L * \varphi_{|L|^{-1}}\quad\forall L\in\Lambda_1,\qquad \int_{\R} \left(\sum_{L\in\Lambda_1}|F_L|^2\right)^{1/2}\lesssim \|f\|_{L\log^{1/2}L},
\end{equation}
where $\Delta_L$ is as in \S\ref{sec:hodef} and
\[
\varphi_{\lambda}(x)\coloneqq\lambda^{-1}\varphi(x/\lambda) \coloneqq  \lambda^{-1}(1+ |x/\lambda|^2)^{-3/4},\qquad x\in\R.
\]
Here and throughout the paper we use local Orlicz norms and corresponding notation as described in \S\ref{sec:not:orl}.

There is a dyadic version: denoting by $\mathcal D_k$ the dyadic subintervals of $[0,1]$ of length $2^{-k}$, $k\in\N_0\coloneqq \N\cup\{0\}$, we consider the conditional expectation and martingale differences
\[
\mathbf E_k f\coloneqq \sum_{I\in\mathcal D_k} \langle f\rangle_I \ind_I,\qquad \mathbf D_k f\coloneqq \mathbf E_kf -\mathbf E_{k-1}f,\quad k\geq 1,\qquad \mathbf D_0f \coloneqq \mathbf E_0 f,\quad f\in L^1.
\]
For future reference we record the definition of the dyadic martingale square function
\[
\mathrm{S}_{\mathcal M} f \coloneqq \left( \sum_{k\geq 1} |\mathbf{D}_k f|^2 \right)^{1/2} .
\]
The dyadic analogue of \eqref{eq:TWsf} is that if $f\in L\log^{1/2} L$  then for each $k\in \N_0$ there exist functions $f_k$ such that
\begin{equation}\label{eq:TWsfdyad}
|\mathbf D_k f| \leq \mathbf E_k|f_k| \quad \forall k\in\N_0,\qquad \int_{[0,1]} \left(\sum_{k\geq 0}|f_k|^2\right)^{1/2}\lesssim \|f\|_{L\log^{1/2}L},
\end{equation}
In fact, the authors in \cite{TW} first prove \eqref{eq:TWsfdyad} by constructing the functions $f_k$ through a rather technical induction scheme, and then deduce \eqref{eq:TWsf} from \eqref{eq:TWsfdyad} via a suitable averaging argument. 

Several remarks are in order. Firstly one notices that \eqref{eq:TWsfdyad} combined with a simple duality argument based on the fact that $\exp (L^2)=(L\log^{1/2}L)^*$ implies the Chang--Wilson--Wolff inequality
\begin{equation}\label{eq:CWW}
\| f-\mathbf E_0f\|_{\exp (L^2)}\lesssim \| \mathrm{S}_\mathcal{M} f \|_{L^\infty} .
\end{equation}
Estimate \eqref{eq:CWW} was first proved in \cite{CWW}; see also the monograph \cite{Wilson} for an in-depth discussion of exponential square integrability in relation to discrete and continuous square functions in analysis. Thus the proof of \eqref{eq:TWsfdyad} in \cite{TW} is of necessity somewhat hard as it reproves \eqref{eq:CWW}. 

A second observation that goes back to \cite{TW}, see also \cite{ST} for an analogous remark on the dual side, is that \eqref{eq:TWsf} implies the weaker estimate
\begin{equation}\label{eq:genzyg}
\left(\sum_{L\in\Lambda_1^1 } \left\|\Delta_L f\right\|_{L^1} ^2 \right)^{\frac 12}\lesssim \|f\|_{L\log^{1/2}L}.
\end{equation}
Indeed, \eqref{eq:genzyg} follows by \eqref{eq:TWsf} and the Minkowski integral inequality. Alternatively, as observed in \cite{ST}, the dual of \eqref{eq:genzyg} is a ---again weaker--- consequence of the Chang--Wilson--Wolff inequality \eqref{eq:CWW}. 

Finally, a consequence of \eqref{eq:genzyg} is the Zygmund inequality  
\[
\left(\sum_{\lambda \in\mathrm{lac}_1 ^1} \left| \widehat f (\lambda) \right|^2\right)^\frac{1}{2}\lesssim \|f\|_{L\log^{1/2}L}.
\]
See for example \cite[{Theorem 7.6, Chapter XII}]{Zyg}. Indeed, if $L_\lambda$ is an interval which has $\lambda$ as an endpoint we have $|\widehat  f(\lambda)|\leq \|(\Delta_{L_\lambda} f)^\wedge\|_{L^\infty}\leq \|\Delta_{L_\lambda} f\|_{L^1}$ for a suitable choice of symbol in the definition of the Littlewood--Paley projection and Zygmund's inequality follows by \eqref{eq:genzyg}. 

All the estimates above have a higher order counterpart which plays an important role in our investigations in this paper. However, our point of view is somewhat different than in \cite{TW}. Firstly we want to emphasize that the proof of our main theorem, Theorem~\ref{thm:main}, hinges on a higher order version of the generalized Zygmund inequality \eqref{eq:genzyg} which loosely has the form
\begin{equation}\label{eq:genzygtauloose}
\left(\sum_{L\in\Lambda_ \tau ^1} \left\|\Delta_{L} f\right\|_{L^1} ^2 \right)^{\frac 12}\lesssim \|f\|_{L\log^{\tau/2}L},\qquad \tau\in \N.
\end{equation}
Estimates of the form \eqref{eq:genzygtauloose} will be referred to as \emph{generalized Zygmund--Bonami inequalities} and will be stated precisely and proved in Section~\ref{sec:genzyg}. The terminology comes from the fact that they imply the higher order version of Zygmund's inequality, due to Bonami \cite{Bonami}, and which can be stated as follows:
\begin{equation}\label{eq:Bonami}
\left(\sum_{ \lambda \in \mathrm{lac}_\tau ^1} \left| \widehat f (\lambda) \right|^2\right)^\frac{1}{2}\lesssim \|f\|_{L\log^{\tau/2}L},\qquad \tau \in \N.
\end{equation}

A novelty in our approach is the realization that the weak square function characterization \eqref{eq:TWsfdyad} of the space $L\log^{1/2}L$, in the dyadic case, is precisely the dual estimate of the Chang--Wilson--Wolff inequality \eqref{eq:CWW}. This relies on a duality argument involving quotient spaces which is inspired by the work of Bourgain, \cite{Bou}. We can then use the Chang--Wilson--Wolff inequality for general order of integrability, see \S\ref{sec:proofthmscnd},
\begin{equation}\label{eq:CWWgen}
\| f-\mathbf E_0f\|_{\exp (L^{2/(\sigma+1)})}\lesssim \| S_\mathcal{M} f \|_{\exp(L^{2/\sigma})},\qquad \sigma \geq 0,
\end{equation}
to conclude the following weak square function characterization of the space $L\log^{ (\sigma+1)/2}L$ in the form of the following theorem.

\begin{theorem}\label{thm:second}If $f\in L\log^{(\sigma+1)/2} L$ for some $\sigma\geq 0$ then for each $k\in\N_0$ there exist functions $f_k$ such that
\[
\mathbf D_k f = \mathbf D_k f_k  \quad \forall k\in\N_0,\qquad \left\| \left(\sum_{k\geq 0}|f_k|^2\right)^{1/2}\right\|_{L\log^{\sigma/2} L}\lesssim \|f\|_{L\log^{(\sigma+1)/2}L}.
\]
The implicit constant depends only on $\sigma$.
\end{theorem}

We will prove Theorem~\ref{thm:second} in Section~\ref{sec:SFCtau} as a consequence of \eqref{eq:CWWgen}. While this is a rather deep implication, as in the case $\sigma=0$, it is not hard to see that the conclusion of Theorem~\ref{thm:second} combined with the fact $\exp(L^{2/\sigma})=(L\log^{\sigma/2}L)^*$ actually implies the Chang--Wilson--Wolff inequality \eqref{eq:CWWgen} for the same value of $\sigma$. We note that while the conclusion of Theorem~\ref{thm:second} and of the subsequent corollary below are already in \cite{TW} for the case $\sigma=0$, our approach provides an alternative proof even for $L\log^{1/2}L$. This approach has the advantage of being able to deal with all spaces $L\log^{(\sigma+1)/2}L$ at once, hence leading to the more general conclusion of Theorem~\ref{thm:second}.

As in the case $\sigma=0$, Theorem~\ref{thm:second} readily implies the continuous version below.

\begin{corollary}\label{cor:SFCtau} Let $J\subset \R$ be a finite interval,  $\sigma\geq 0$ and $f\in L\log^{(\sigma+1)/2} L(J)$. Then for each $L\in \Lambda_1 $ with $|L|\geq |J|^{-1}$ there exists a nonnegative function $F_L$ such that for every $\gamma \geq 1$
\[
|\Delta_L f|\lesssim F_L * \varphi_{(|L||J|)^{-1}},\quad \left\| \left(\sum_{L\in\Lambda_1^{|J|^{-1}}}|F_L|^2\right)^{1/2}\right\|_{{L\log^{\sigma/2}L\left(\gamma J,\tfrac{\d x}{|J|}\right)}}\lesssim \|f\|_{L\log^{(\sigma+1)/2}L\left(J,\tfrac{\d x}{|J|}\right)},
\]
with implicit constant depending only on $\gamma$ and $\sigma$. If in addition  $\int_J f=0$ then the conclusion holds for all $L\in\Lambda_1$ with the summation extending over all $L\in\Lambda_1$. With or without this additional assumption, for $|L|\geq |J|^{-1}$ the functions $F_L$ are supported in $5J$. The implicit constant depends only on $\gamma$ and $\sigma$, as indicated.
\end{corollary}

\subsection{Background and history} The fact that Marcinkiewicz multipliers are $L^p$-bounded is classical; see for example 
\cite[Theorem 8.13]{duo}. The first endpoint result concerning multiplier operators of Marcinkiewicz-type is arguably a theorem due to Bourgain \cite{Bou} which asserts that, in the periodic setting, the classical Littlewood--Paley square function $\mathrm {LP}_1$ has operator norm $\|\mathrm {LP}_1\|_{p\to p}\simeq (p-1)^{-3/2}$ as $p\to 1^+$. Tao and Wright proved in \cite{TW} the optimal local endpoint estimate $L\log^{1/2}L \to L^{1,\infty}$ for the class of $R_2=R_{2,1}$ multipliers, which contains Marcinkiewicz multipliers. It was later observed in \cite{Bakas} that Bourgain's estimate follows by the endpoint bound of \cite{TW} combined with a randomization argument and Tao's converse extrapolation theorem from \cite{Tao}. Recently, Lerner proved in \cite{Lerner} effective weighted bounds for the classical Littlewood--Paley square function $\mathrm{LP}_1$; these weighted bounds imply the correct $p$-growth for the $L^p\to L^p$ norms of these operators as $p\to 1^+$. In addition, as observed in \cite{Bakas_L^p}, the arguments of \cite{Lerner} can be used to establish weighted $A_2$ estimates for $\mathrm{LP}_{\tau}$ that imply sharp $L^p\to L^p$ estimates for $\mathrm{LP}_{\tau}$ as $p\to 1^+$ for any order $\tau$. The class $R_2$ contains all multipliers $m$ whose pieces $m_L$ have bounded $q$-variation uniformly in $L\in\Lambda_1$, for all $1\leq q< 2$; see \cite{CRS} where the authors showed that all $R_2$ multipliers are bounded on $L^p$ for $p\in(1,\infty)$.

As already discussed, the authors in \cite{TW} rely on the weak square function characterization of $L\log^{1/2}L$ as in \eqref{eq:TWsf} for their proof. Our argument here is a bit different, relying on the weaker generalized Zygmund--Bonami inequality instead; a hint of a different proof already appears in \cite[p. 540]{TW}. The Zygmund inequality first appeared in \cite{Zygpap} in its dual form; see also \cite[{Theorem 7.6, Chapter XII}]{Zyg} The higher lacunarity order \eqref{eq:Bonami} is due to Bonami and it is contained in \cite{Bonami}. We note that our results provide an alternative proof for the case of finite order lacunary sets. On the other hand, a dual version of the generalized Zygmund--Bonami inequality in the first order case (that is, inequality \eqref{eq:genzyg}) appears in \cite{ST}.

The $L^p$-boundedness of Marcinkiewicz multipliers of order one and higher in the periodic setting was established by Marcinkiewicz in \cite{M}; see also Gaudry's paper \cite{Gaudry}. Generalized versions of H\"ormander-Mihlin and Marcinkiewicz multipliers, together with their square function counterparts of higher order, have been introduced in \cite{SjSj} in a very broad context. There the authors proved the equivalence of $L^p$-boundedness between different classes of such multipliers. Our setup is focused on the finite order lacunary case and provides the optimal endpoint bounds for such classes.

\subsection{Structure} The general structure of the rest of this paper is as follows. Section~\ref{sec:not} contains some basic facts and properties of Orlicz spaces, together with a small toolbox for dealing with lacunary sets; the reader is encouraged to skip this section on a first reading and only consult it when necessary. In Section~\ref{sec:SFCtau} we will prove Theorem~\ref{thm:second} and Corollary~\ref{cor:SFCtau}. In Section~\ref{sec:genzyg} we will critically use Corollary~\ref{cor:SFCtau} in order to conclude the generalized Zygmund--Bonami inequality of arbitrary order alluded to above. This inequality will be stated and proved in different versions which can be local or non-local, depending on the type of cancellation assumptions we impose. The reader can find the corresponding statements in Propositions~\ref{prop:genzygpos} and ~\ref{prop:genzygcanc}; see also Corollary~\ref{cor:genzygbonami}. In Section~\ref{sec:CZdecomp} we present the details of a Calder\'on-Zygmund decomposition for the Orlicz space  $L\log^{\sigma/2}L$, adapted to the needs of this paper. The proof of Theorem~\ref{thm:main} takes up the best part of Section~\ref{sec:proofmain} where the Calder\'on-Zygmund decomposition of Section~\ref{sec:CZdecomp} is combined with the generalized Zygmund--Bonami inequality of Section~\ref{sec:genzyg}. The proofs of Theorem ~\ref{cor:LPsf} and Theorem~\ref{thm:horm} are discussed in Section~\ref{sec:cormult} as a variation of the proof of Theorem~\ref{thm:main}. 

\section{Preliminaries and notation}\label{sec:not}In this section we collect several background definitions and notations that will be used throughout the paper.

\subsection{Some basic facts for certain classes of Orlicz spaces} \label{sec:not:orl}  We adopt standard nomenclature for Young functions and Orlicz spaces as for example in \cite[{Chapter 10}]{Wilson}. Given a Young function $\Phi:[0,\infty]\to [0,\infty]$ we will use the following notation for local $L^\Phi$ averages: For a finite interval $I\subset \R$ 
\[
\langle |f|\rangle_{\Phi,I}\coloneqq \inf\left\{\lambda>0:\, \frac{1}{|I|}\int_I \Phi\left(\frac{|f(x)|}{\lambda}\right)\,\d x\leq 1\right\}.
\]
For the usual local $L^p$ averages we just set $\langle|f|\rangle_{p,I}\coloneqq |I|^{-1/p}\|f\|_{L^p(I)}$ for $1\leq p<\infty$. For $\sigma\geq 0$ we use the Young function $B_\sigma(t)\coloneqq t(\log(e+t))^\sigma$ to define local $L\log^\sigma L$-spaces and we will also write
\begin{equation}\label{eq:OrliczAvg}
\|f\|_{L\log^\sigma L{\left(I,\tfrac{\d x}{|I|}\right)}}\coloneqq \langle |f |\rangle_{B_\sigma,I}\simeq\frac{1}{|I|}\int_I |f(x)|  \left(\log\left(e+\frac{|f(x)|}{\langle|f|\rangle_{1,I}}\right)\right)^\sigma\,\d x.
\end{equation}
The last approximate equality can be found in \cite[Theorem 10.8]{Wilson}. For future reference it is worth noting that the function $B_\sigma$ is submultiplicative and thus doubling; see \cite[{\S 5.2}]{UMP}. We will write instead $L^{B_{\sigma}}(\R)$ to denote the (global) space of measurable functions $f$ such that $\int_{\R}B_{\sigma}(|f|)<+\infty.$

The dual Young function of $B_\sigma$ can be taken to coincide with $E_{\sigma^{-1}}(t)\coloneqq \exp(c_\sigma t^{1/\sigma})-1$ for $t\gtrsim 1$; here we insist on the equality only for sufficiently large  values of $t$; with this function we define the local $\exp(L^{1/\sigma})$ norms and we have the H\"older inequality $\langle|fg|\rangle_{1,I}\lesssim \langle |f|\rangle_{B_\sigma,I}\langle |f|\rangle_{E_{\sigma^{-1}, I}}$. We reserve the notation $L\log^\sigma L$ and $\exp(L^{1/\sigma})$ for the case $I=[0,1]$ and the space of functions supported in $[0,1]$ for which
\[
\|f\|_{L\log^\sigma L} \coloneqq \langle|f|\rangle_{B_\sigma,[0,1]}<+\infty,\quad \|f\|_{\exp{(L^{1/\sigma})}}\coloneqq \langle|f|\rangle_{E_{\sigma^{-1}},[0,1]}\simeq\sup_{p\geq 2}p^{-\sigma}\|f\|_p   <+ \infty,
\]
respectively; see \cite[{\S2.2.4}]{Tri} for the last approximate equality; We will also use the well known duality identification $ (L\log^\sigma L)^*\cong \exp(L^{1/\sigma})$. For $\sigma=0$ we adopt the convention that $L\log^\sigma L=L^1$ and $\exp(L^{1/\sigma})=L^\infty$. The following Minkowski-type integral inequality
\[
\left\| \{\|f_k\|_{L\log^\sigma L}\}\right\|_{\ell^2 _k}\lesssim \left\| \| \{f_k\}\|_{\ell^2 _k} \right\|_{L\log^\sigma L}
\]
 will be used with no particular mention. Its proof can be obtained by a simple duality argument.

\subsection{Some tools for handling lacunary sets}\label{sec:lacnot} We introduce some useful notions concerning lacunary sets of arbitrary order. Let $\tau\geq 1$ and $L\in \Lambda_\tau$. We will denote by $\widehat{L}$ the unique interval $\widehat L\in\Lambda_{\tau-1}$ such that $L\subset \widehat L$ and call $\widehat L$ the (lacunary) \emph{parent} of $L$. Furthermore we will denote by $\lambda(L)$ the unique element $\lambda\in\mathrm{lac}_{\tau-1}$ such that $\dist(L,\R\setminus \widehat L)=\dist(L,\lambda)=|L|$. We note that $\lambda(L)$ is one of the endpoints of $\widehat{L}$. These definitions also make sense in the case $\tau=1$ remembering the definitions of $\Lambda_0$ and $\mathrm{lac}_0$.

If $L\in \Lambda_\tau$ then ${L^*}\coloneqq L-\lambda(L) \in\Lambda_1$; in fact $L^*$ is one of the intervals $(-2|L|,-|L|)$ or $(|L|,2|L|)$ depending on the original relative position of $L$ with respect to $\lambda(L)$. The point of the definitions above is that if $L\in\Lambda_\tau$ then, upon fixing a suitable choice of bump functions $\phi_L \in \Phi_{L,M}$, we can write the identity
\[
\Delta_L f =e^{2\pi i \lambda(L)\bigcdot} \Delta_{L^*}\left(e^{-2\pi i \lambda(L)\bigcdot}\Delta_{\widehat L}f\right)=e^{2\pi i \lambda(L)\bigcdot} \Delta_{L^*}\left(e^{-2\pi i \lambda(L)\bigcdot}f\right).
\]
This will be crucially used in several parts of the recursive arguments in the paper. We will also use the intuitive notation $\Delta_{|L|}\coloneqq \Delta_{L^*\cup(-L^*)}$ for the smooth Littlewood--Paley projection of first order at frequencies $|\xi|\simeq |L|$ which takes advantage of the fact that $L^*$ essentially only depends on the length of $L$. The following notation will be useful to localize in a certain lacunary parent:
\[
\Lambda_\tau ^n (L')\coloneqq\{L\in\Lambda_\tau ^n:\, L\subset L'\},\qquad L'\subset \R;
\]
similarly we define $\Lambda_\tau (L')$. Note that if $L'\in\Lambda_{\tau-1}$ and $L\in\Lambda_\tau(L')$ then necessarily $\widehat{L}=L'$.

The following simple lemma relies on the fact that lacunary sets are invariant under dyadic dilations with respect to the origin and will be used to allow rescaling of intervals of dyadic length to $[0,1]$. 
\begin{lemma}\label{lem:dilation}Let $\tau\in\N$ and $a\in 2^\Z$. Then $a^{-1}\mathrm{lac}_\tau ^a \coloneqq \{\lambda/a:\, \lambda\in\mathrm{lac} _\tau ^a\}= \mathrm{lac}_{\tau} ^1.$
\end{lemma}
To showcase the typical application of this lemma let $J\subset \R$ be an interval of dyadic length and $a=\{a_\lambda\}_{\lambda\in \mathrm{lac}_\tau}$  a finite collection of complex coefficients. By a standard change of variables
\begin{equation}\label{eq:rescaling}
p_a(y)\coloneqq \sum_{\lambda\in\mathrm{lac}_\tau ^{|J|^{-1}}} a_{\lambda}e^{i\lambda y},\qquad \langle|p_a|\rangle_{p,J} ^p= \int_{[0,1]}\bigg|\sum_{\lambda\in|J|\mathrm{lac}_\tau ^{|J|^{-1}}} a_{\lambda|J|^{-1}}e^{i\lambda y} \bigg|^p\, \d y ,
\end{equation}
and we crucially note that the sum on the right hand side is for $\lambda \in|J|\mathrm{lac}_\tau ^{|J|^{-1}}=\mathrm{lac}_\tau ^1$ because of the lemma. Of course the same change of variables will be valid  for $\langle|p_a|\rangle_{\Phi,J}$ for any Young function $\Phi$. We will use this rescaling argument in several places in the paper.

\subsection{Other notation} For any function $g$ and $\lambda>0$ we write $g_\lambda(x)\coloneqq \lambda^{-1}g(x/\lambda)$ for the $L^1$-rescaling. Two special kinds of bump functions will appear. Firstly $\omega(x)\coloneqq (1+|x|^2)^{-N/2}$ is the smooth tailed indicator of $[-1/2,1/2]$ with $N$ any large positive integer. It will be enough to take $N=10$ for the arguments in this paper but more decay is available if needed. We will also write $\varphi (x)\coloneqq  (1+|x|^2)^{-3/4}$ which is still an $L^1$-bump but has only moderate decay. In some cases we are restricted to using $\varphi$, most notably in the statement and proof of Corollary~\ref{cor:SFCtau}.

\section{A weak square function characterization of \texorpdfstring{$L\log^{\sigma/2}L$}{LlogL}}\label{sec:SFCtau} In this section we provide the proof of Theorem~\ref{thm:second} as a consequence of the Chang--Wilson--Wolff inequality of general order \eqref{eq:CWWgen}. The conclusion of Corollary~\ref{cor:SFCtau} will then follow by a standard averaging argument using almost orthogonality between the continuous Littlewood--Paley projections and martingale differences.

\subsection{Proof of Theorem \ref{thm:second}}\label{sec:proofthmscnd} We recall that we work on the probability space $([0,1],\d x)$. It clearly suffices to prove the theorem for $k\geq 1$ as for $k=0$ we can set $f_0\coloneqq \mathbf E_0 f=\mathbf D_0 f$. Our starting point is the Chang--Wilson--Wolff inequality of general order of integrability, \eqref{eq:CWWgen}. This is pretty standard but a quick proof can be produced by using the usual Chang--Wilson--Wolff inequality \eqref{eq:CWW} in the form
\[
p^{-\sigma/2}\| f-\mathbf E_0 f\|_p \lesssim p^{-\sigma/2} p^{1/2} \|S_{\mathcal M}f\|_p,\qquad p\geq 2,\quad \sigma\geq 0,
\]
which readily implies
\[
\|f-\mathbf E_0 f\|_{\exp(L^{2/(\sigma+1)})}\simeq\sup_{p\geq 2} \frac{\|f-\mathbf E_0 f\|_p}{p^{(\sigma+1)/2}}\lesssim \sup_{p\geq 2} \frac{\|S_{\mathcal M}f\|_p}{p^{\sigma/2}}\simeq \|S_{\mathcal M}f\|_{\exp(L^{2/\sigma})}
\]
which is \eqref{eq:CWWgen}. Observe that \eqref{eq:CWWgen} has the form
\begin{equation}\label{eq:interdual}
\left\| \sum_{k\geq 1}  g_k\right \|_{\exp(L^{2/(\sigma+1)})} \lesssim \left\| \left(\sum_{k\geq 1}|g_k|^2\right)^{1/2} \right\|_{\exp(L^{2/\sigma})},\qquad g_k =\mathbf D_k f.
\end{equation}
We will write \eqref{eq:interdual} as a continuity property for the operator $\mathrm{T}(\{g_k\}_k)\coloneqq \sum_k g_k$ between suitable Banach spaces. To that end let us consider the subspace of $L\log^{\sigma/2}L([0,1];\ell^2)$ given by 
\[
 Y \coloneqq \left\{ \{\psi_k\}_k \in L\log^{\sigma/2}L([0,1];\ell^2):\, \mathbf{D}_k \psi_k = 0 \, \,\text{ for all } \,\, k\in\N \right\}. 
\]
We observe that $Y$ is closed. To see this consider a sequence $(\psi^n)_n\subset Y$ with $\psi^n= \{\psi_k ^n \}_k$ converging to some $\psi=\{\psi_k\}_k$ in $L\log L^{\sigma/2}L([0,1];\ell^2)$. Clearly the limit $\psi$ belongs to the space $L\log L^{\sigma/2}L([0,1];\ell^2)$, the latter being a Banach space and, additionally, $\psi_k^n$ converges to $\psi_k$ in $L\log ^{\sigma/2}L([0,1])$ and so also in $L^1([0,1])$, uniformly in $k$. Now it follows by Fatou's lemma that for each $k\in\N$ there holds
\[
\left\|\liminf_{n\to \infty} \left|\mathbf{D}_k(\psi_k ^n -\psi_k)\right|\right\|_{L^1([0,1])}\leq\liminf_{n\to \infty} \left\| \psi_k - \psi_k ^n\right\|_{L^1([0,1])}=0
\]
yielding $\mathbf{D}_k \psi_k ^n = \mathbf{D}_k\psi_k=0$ a.e., where we also used the uniform boundedness of $\mathbf{D}_k$ on $L^1([0,1])$.

Since $\big(L\log^{\sigma/2} L([0,1];\ell^2)\big)^{\ast} \cong \exp(L^{2/\sigma})([0,1];\ell^2)$, the annihilator of $Y$ is given equivalently by 
\[ 
Y^{\perp} = \left\{ \{g_k\}_k \in \exp(L^{2/\sigma})([0,1];\ell^2) :\, \int \left(\sum_k g_k \psi_k \right) =0\,\, \text{ for all } \,\, \{\psi_k\}_k \in Y \right\}.
\]
Since $Y$ is a closed subspace of $L\log^{\sigma/2}L([0,1];\ell^2)$ we have that $(L\log^{\sigma/2}L([0,1];\ell^2)/Y)^\ast$ is isometrically isomorphic to $Y^\perp$; see \cite[{Theorem 4.9}]{Ru}.  We equip $Y^\perp$ with the norm appearing on the right hand side of \eqref{eq:interdual}. We will use the following fact.

\begin{lemma}\label{lem:g_k_projection}
If $\{g_k\}_k \in Y^{\perp}$ then  $\mathbf{D}_k g_k = g_k$ for every $k\in\N$.
\end{lemma}

\begin{proof}
Fix an index $k_0\in\N$, let $\psi \in L\log^{\sigma/2}L$ be arbitrary and let $\{\psi_k\}_k$ be defined by 
\[ 
\psi_k \coloneqq \begin{cases} 
0 & \text{if } k \neq k_0, 
\\
\psi - \mathbf{D}_{k_0} \psi & \text{otherwise}.
\end{cases} 
\]
Clearly $\{\psi_k\}_k \in L\log^{\sigma/2}L([0,1];\ell^2)$ and, moreover, $ \mathbf{D}_{k_0}(\psi - \mathbf{D}_{k_0} \psi) = 0$ so that $\{\psi_k\}_k \in Y$. By the definition of $Y^{\perp}$ we have then
\[ 
0 = \int \sum_{k} g_k \psi_k  = \int g_{k_0} (\psi - \mathbf{D}_{k_0} \psi)  = \int (g_{k_0} - \mathbf{D}_{k_0} g_{k_0}) \psi
\]
where we have used the fact that $\mathbf{D}_k$ is self-adjoint; but this is only possible for arbitrary  $\psi$  if $\mathbf{D}_{k_0} g_{k_0} = g_{k_0}$, as claimed.
\end{proof}

Now \eqref{eq:interdual} can be written in the form
\begin{equation}\label{eq:Tgk}
\left \|\mathrm T(\{g_k\}_k)\right\|_{\exp(L^{2/(\sigma+1)})}\lesssim \left\| \{g_k\}_k\right\|_{Y^\perp},\qquad \mathrm T(\{g_k\}_k)\coloneqq \sum_k g_k.
\end{equation}
Let $X\coloneqq L\log^{\sigma/2}L([0,1];\ell^2)$ and denote by $X_N, Y_N$ the functions in $X,Y$, respectively, which are constant on dyadic intervals of length smaller than $2^{-N}$. In particular, such functions $f$ have finite Haar expansion which implies the apriori qualitative property that the spaces $X_N,Y_N$ are finite dimensional. We note that $(X_N/Y_N)^\ast$ is isometrically isomorphic to $Y_N ^\perp$. By the Riesz representation theorem we then get that
\[
\begin{split}
\|\{\mathbf D_k f\}_k\|_{X_N/Y_N} &\leq \sup_{\substack{\{g_k\}_k \in Y_N ^{\perp} :\\ \|\{g_k\}_k\|_{Y_N ^\perp}\leq 1}} \left| \int \sum_{k} g_k \mathbf{D}_k f  \right|
=\sup_{\substack{\{g_k\}_k \in Y_N ^{\perp} :\\ \|\{g_k\}_k\|_{Y_N ^\perp}\leq 1}} \left| \int \mathrm T(\{g_k\}_k)f \right|,
\end{split}
\]
where we also used Lemma~\ref{lem:g_k_projection} in passing to the equality in the right hand side above. Using \eqref{eq:Tgk} together with H\"older's inequality in Orlicz spaces it follows that
\[
\|\{\mathbf D_k f_N\}_k\|_{X/Y}=\inf_{\{\psi_k\}_k \in Y} \left\| \left(\sum_{k} |\mathbf{D}_k f_N + \psi_k|^2\right)^{1/2} 
\right\|_{L\log^{\sigma/2}L}\lesssim \|f_N\|_{L\log^{(\sigma+1)/2}L},
\]
where $f_N$ is the truncation of the Haar series of $f\in X$ at scale $2^{-N}$. We stress that the approximate inequality above holds uniformly for all $N\in\N$. This inequality extends to all $f \in L\log^{(\sigma+1)/2}L$ by a standard approximation argument, using the fact that the truncated Haar series of functions $f \in L\log^{(\sigma+1)/2}L$ converge to $f$ in $L\log^{(\sigma+1)/2}L$; see \cite{OSW83}. The extension of the operator $f \mapsto \{\mathbf D_k f\}_k$ is the obvious one given by the same expression. 

In order to conclude the proof of the theorem we notice that for every $\{\psi_k\}_k\in Y$ there holds $\mathbf D_k f=\mathbf D_k (\mathbf D_k f +\psi_k)$ and the last inequality guarantees the existence of a vector $\{\psi_k\}_k\in Y$ such that the functions $f_k\coloneqq \mathbf D_k f +\psi_k$ satisfy the conclusion of the theorem.

\subsection{Proof of Corollary~\ref{cor:SFCtau}} The corollary follows from the dyadic case of Theorem~\ref{thm:second} via an averaging argument which is essentially identical to the one in \cite[{\S9}]{TW}; {see also \cite{MarcoV} where the argument of \cite[{\S9}]{TW} is explained in detail.
 We can clearly assume that $\gamma\geq 3$ and by affine invariance we can take $\gamma J=[0,1]$, so that $|J|=\gamma^{-1}$, and $\supp f\subset[1/3,2/3]$.

For all $\theta\in [-1/3,1/3]$ we define $f_\theta(x):=f(x-\theta)$. By Theorem~\ref{thm:second}, for each $\theta\in [-1/3,1/3]$ and $k\geq 0$ there exists a function $f_{\theta,k}$ such that $\mathbf{D}_k f_\theta=\mathbf{D}_k f_{\theta,k}$ and
\begin{align}\label{proof_generalization_4.1_square_function_fktheta}
\left\| \left( \sum_{k\geq 0} |f_{\theta,k}|^2 \right)^{1/2} \right\|_{L \log^{\sigma/2}L}\lesssim \|f_\theta\|_{L\log^{(\sigma+1)/2}L}= \|f\|_{L \log^{(\sigma+1)/2} L}.
\end{align}
Setting for $L\in\Lambda_1 ^\gamma$ 
\[
F_L(x)\coloneqq \sum_{k\in \mathbb{N}_0} 2^{-|\log_2|L|-k|/2}\int_{[-1/3,1/3]}
|f_{\theta,k}(x+\theta)| \,\d\theta
\]
and arguing as in \cite{TW} we see that $|\Delta_L f(x)|\lesssim F_L\ast \varphi_{|L|^{-1}}$ and
\[
\bigg( \sum_{L\in\Lambda_1 ^\gamma} |F_L(x)|^2 \bigg)^{1/2} \lesssim
\int_{[-1/3,1/3]}\bigg( \sum_{k\geq 0} |f_{\theta,k} (x+\theta)|^2 \bigg)^{1/2} \,\d\theta.
\]
By the Minkowski integral inequality for the space $L\log^{\sigma/2}L$ we have
\[    
\left\|   \left( \sum_{L\in\Lambda_1 ^\gamma} |F_L |^2 \right)^{1/2} \right\|_{L\log^{\sigma/2}L} \lesssim \int_{[-1/3,1/3]} \left\| \left( \sum_{k\geq 0} |f_{\theta,k} ( \cdot+\theta)|^2 \right)^{1/2}\right\|_{L\log^{\sigma/2}L}  \, \d\theta 
\]
and the proof follows for $L\in\Lambda_1^\gamma$. Under the additional cancellation assumption $\int_{[0,1]} f=0$ we consider also $L\in\Lambda_1$ with $|L|<\gamma$, and for these we define $F_L\coloneqq |\Delta_L f|$ and note that $|\Delta_L f|\lesssim \varphi_{|L|^{-1}} * F_L$. Using the cancellation condition we have also
\[
|\Delta_L f| \lesssim  \varphi_{|L|^{-1}} *(|L|\|f\|_{L^1}\ind_{[0,1]}),
\]
which readily yields the estimate $\left\|\|\{F_L\}_{|L|<\gamma}\|_{\ell^2}\right\|_{L^1}\lesssim \|f\|_{L^1}$ and the proof is complete. Note that we used that since $\tau=1$ there are at most two intervals $L\in\Lambda_1$ of any given length.

\section{Generalized Zygmund--Bonami inequalities} \label{sec:genzyg} In this section we prove  the versions  of the generalized Zygmund--Bonami inequality presented in the introduction,  where the Littlewood--Paley projections $\Delta_L$ for $L\in\Lambda_1$ are replaced by their $\tau$-order counterparts $\Delta_{L}$ for $L\in \Lambda_\tau$, where $\tau$ is an arbitrary positive integer. As already discussed, the estimate corresponding to order $\tau=1$ is \eqref{eq:genzyg} and it follows rather easily from the case $\sigma=0$ of Corollary~\ref{cor:SFCtau}. For $\tau>1$ we first state the generalized Zygmund-Bonami inequalities in the case of $L\in \Lambda_\tau$ with $|L|\geq |J|^{-1}$, where $J$ is an interval in which $f$ is supported; this is the harder and deeper case. In the rest of the section we will also provide the statements and proofs for the easier case $|L| < |J|^{-1}$; the latter will rely on pointwise estimates for $\Delta_{L}f$ and recursive arguments, assuming suitable cancellation conditions for $f$ in the same spirit as Corollary~\ref{cor:SFCtau}.

\subsection{The main term in the generalized Zygmund--Bonami inequalities} We encourage the reader to keep in mind the notation of \S\ref{sec:not} for the local Orlicz norms and the definitions concerning lacunary sets from \S\ref{sec:lacnot} for the rest of this section. Our first result below gives a version of the generalized Zygmund--Bonami inequality in which the intervals $L$ are restricted to those for which $|L|\geq |J|^{-1}$, as anticipated above.

\begin{proposition}\label{prop:genzygpos}Let $J\subset \R$ be a finite interval and $f$ be a compactly supported function with $\mathrm{supp}(f)\subseteq J$. Let $\tau$ be a positive integer, $\sigma\geq 0$ 
and $\gamma > 1$. There holds
\[
\left( \sum_{L\in \Lambda_\tau ^{|J|^{-1}}}  \left\langle|\Delta_{L} f |\right 
\rangle_{B_{\sigma/2},\gamma J} ^2 \right)^{1/2}
\lesssim_{\sigma,\tau,\gamma} \l |f|\r _{B_{(\sigma+\tau)/2}, J}   \]
and
\[ \sum_{L\in\Lambda_\tau ^{|J|^{-1}}} \left\| \Delta_{L}f\right\|_{L^2(\R\setminus \gamma J)} ^2 \lesssim_{\tau,\gamma} |J| \langle |f|\rangle_{B_{(\tau-1)/2},J} ^2.
\]
\end{proposition}

\begin{proof} The proof is by way of induction on $\tau$, with the base case $\tau=1$ being an easy consequence of Corollary~\ref{cor:SFCtau}, as we shall now illustrate. Indeed, let $C_1(\sigma,\tau,\gamma)$ and $C_2(\tau,\gamma)$ denote the best constants in the first, and the second estimate in the statement, respectively. Corollary~\ref{cor:SFCtau} implies that for $L\in\Lambda_1 ^{|J|^{-1}}$ we have
\[
\langle |\Delta_L f| \rangle _{B_{\sigma/2},\gamma J}\lesssim \langle |\varphi_{(|L||J|)^{-1}}*F_L| \rangle _{B_{\sigma/2},\gamma J}\lesssim \langle |F_L|\rangle _{B_{\sigma/2}, \gamma J},
\]
using Young's convolution inequality and the $L^1$-normalization of each $\varphi_{(|L||J|)^{-1}}$. Now the proof of the first estimate in the conclusion for $\tau=1$ can be concluded by yet another application of Minkowski's inequality, this time to yield that the left hand side of the first estimate in the conclusion is bounded by a constant multiple of
\[
\left(\sum_{|L|\geq |J|^{-1}}  \l |F_L| \r  _{B_{\sigma/2},\gamma J} ^2\right)^{1/2}\lesssim \left\langle \left( \sum_{|L|\geq |J|^{-1}}|F_L|^2\right)^{1/2}\right\rangle _{B_{\sigma/2},\gamma J}\lesssim \l |f|\r _{B_{(\sigma+1)/2},J} 
\]
by the estimate for the square function of the $\{F_L\}_{L}$ in Corollary~\ref{cor:SFCtau}. Thus $C_1(\sigma,1,\gamma)<+\infty$ for all nonnegative integers $\sigma$ and $\gamma >1 $.

For the second estimate we have for $|L|\geq |J|^{-1}$ and $x\in \R\setminus \gamma J$
\begin{equation}\label{eq:delta}
|\Delta_Lf(x)| \lesssim  |L| |J| (1+|L||x-c_J|)^{-10}\l|f|\r_{1,J} \lesssim \omega_{|L|^{-1}}*(\l|f|\r_{1,J}\ind_J)(x),
\end{equation}
with $\omega_{|L|^{-1}}$ as given in \S\ref{sec:not}. Using the first approximate inequality above, the square of the left hand side of the second estimate in the conclusion of the proposition can be estimated by a constant multiple of 
\[
\sum_{L\in\Lambda_1 ^{|J|^{-1}}} \langle |f|\rangle_{1,J} ^2 (|L||J|)^2|L|^{-20} \int_{\R\setminus \gamma J}|x-c_J|^{-20}\,\d x\lesssim \langle |f|\rangle_{1,J} ^2 \sum_{L\in\Lambda_1:\, |L||J|\geq 1} (|L||J|)^{-18}  |J|
\]
which sums to the desired quantity since, for $\tau=1$, there is exactly one interval $L\in\Lambda_1$ per dyadic scale. This shows that $C_2(1,\gamma)<+\infty$ for all $\gamma>1$.

Consider now the case $\tau>1$ and let $a>1$ be such that $\gamma = a^\tau$. Recalling the discussion in \S\ref{sec:lacnot} we write
\begin{equation}\label{eq:split}
\begin{split}
|\Delta_{L}f|&\leq \left|\Delta_{L^*}\left(\ind_{a^{\tau-1} J} e^{-2\pi i \lambda(L)\bigcdot}\Delta_{\widehat{L}} f\right)\right|+\left|\Delta_{L^*}\left(\ind_{\R\setminus a^{\tau-1} J} e^{-2\pi i \lambda(L)\bigcdot}\Delta_{\widehat{L}} f\right)\right|
\\
&\eqqcolon |\Delta_{L^*} f_{1,L}|+|\Delta_{L^*} f_{2,L}|.
\end{split}
\end{equation}
For clarity, we remind the reader that $\lambda(L)$ is either of the endpoints of $\widehat{L}$ (depending on the position of $L$) and therefore we can partition the intervals $L$ into two families such that $f_{1,L}, f_{2,L}$ \emph{actually depend only on} $\widehat{L}$ -- this will be relevant below. Fixing for a moment $L'\in\Lambda_{\tau-1} ^{|J|^{-1}}$ we note that for any $L\in \Lambda_\tau ^{|J|^{-1}}(L')$ there holds $\widehat{L}=L'$ and so $|f_{1,L}| = |\Delta_{\widehat{L}} f| \ind_{a^{\tau-1} J}= |(\Delta_{L'} f ) \ind_{a^{\tau-1}J}|$. As the collection $\{L^*:\, L\in \Lambda_{\tau-1} ^{|J|^{-1}}(L')\}\subset \Lambda_1 ^{|J|^{-1}}$ we can use the conclusion of proposition for $\tau=1$ to estimate for fixed $L'\in \Lambda_{\tau-1} ^{|J|^{-1}}$
\[
\begin{split}
 \left(  \sum_{L\in \Lambda_\tau ^{|J|^{-1}}  (L')}  \left\langle|\Delta_{L^*} f_{1,L} |\right 
\rangle_{B_{\sigma/2},a^{\tau} J} ^2 \right)^{1/2} &\leq C_1(\sigma,1,a)  \left\langle| \Delta_{L'}f|\right\rangle_{B_{(\sigma+1)/2},a^{\tau-1}J}  .
\end{split}
\]
Thus we can recursively estimate

\[
\begin{split}
&  \left( \sum_{L' \in \Lambda_{\tau-1} ^{|J|^{-1}}} \sum_{L\in \Lambda_\tau ^{|J|^{-1}}  (L')}  \left\langle|\Delta_{L^*} f_{1,L} |\right 
\rangle_{B_{\sigma/2},a^{\tau} J} ^2 \right)^{1/2}
\\
&\quad\quad \leq C_1(\sigma,1,a) \left(\sum_{L'\in \Lambda_{\tau-1} ^{|J|^{-1}}} \left\langle| \Delta_{L'}f|\right\rangle_{B_{(\sigma+1)/2},a^{\tau-1}J} ^2\right)^{\frac12}
\\
&\quad\quad\leq C_1(\sigma,1,a) C_1  (\sigma+1,\tau-1,a^{\tau-1}) \langle |f|\rangle_{B_{(\tau+\sigma)/2},J} 
\end{split}
\]
which takes care of the contribution of the $f_{1,L}$'s. Considering now the $f_{2,L}$'s, we have by H\"{o}lder's inequality for Orlicz spaces that $\left\langle|\Delta_{L^*} f_{2,L} |\right 
\rangle_{B_{\sigma/2},a^{\tau} J} \leq \left\langle|\Delta_{L^*} f_{2,L} |\right 
\rangle_{2,a^{\tau} J}$ and therefore for any fixed $L' \in \Lambda_{\tau-1} ^{|J|^{-1}}$
\[  \sum_{L\in \Lambda_\tau ^{|J|^{-1}}  (L')}  \left\langle|\Delta_{L^*} f_{2,L} |\right 
\rangle_{B_{\sigma/2},a^{\tau} J} ^2 \lesssim  \frac{1}{a^{\tau}|J|} \int_{\R} |f_{2,L}|^2 = \frac{1}{a^{\tau}|J|} \int_{\R\setminus a^{\tau-1}J} |\Delta_{L'}f|^2,\]
where we have used the $L^2 \to L^2$ boundedness of the smooth Littlewood--Paley square function together with the fact remarked above that $f_{2,L}$ depends essentially only on $\widehat{L} = L'$. It follows that we can bound recursively
\[
\begin{split}
 \left( \sum_{L'\in \Lambda_{\tau-1} ^{|J|^{-1}}} \sum_{\substack{L\in\Lambda_\tau (L')\\|L|\geq |J|^{-1}}}  \left\langle|\Delta_{L^*} f_{2,L} |\right\rangle_{B_{\sigma/2},a^{\tau} J} ^2 \right)^{1/2} &\lesssim_\sigma \left(\frac{1}{a^{\tau}|J|} \sum_{L'\in \Lambda_{\tau-1} ^{|J|^{-1}}}  \int_{\R\setminus a^{\tau-1}J} |\Delta_{L'}f|^2\right)^{1/2}
\\
&\leq \left( \frac{C_2(\tau-1,a^{\tau-1})}{a^{\tau}|J|}|J|\langle |f|\rangle_{B_{(\tau-2)/2},J}^2\right)^{1/2}\\
&\lesssim_{\sigma,\tau} (C_2(\tau-1,a^{\tau-1})/a^{\tau})^{1/2} \langle |f|\rangle_{B_{(\sigma+\tau)/2},J} .
\end{split}
\]
 This proves that 
\[
C_1(\sigma,\tau,a^\tau)\leq  C_1(\sigma,1,a)C_1(\sigma+1,\tau-1,a^{\tau-1})+ c_{\sigma,\tau} C_2(\tau-1,a^{\tau-1})^{1/2} a^{-\tau/2}
\]
for some numerical constant $c_{\sigma,\tau}$ depending only on $\sigma,\tau$.

We move to the proof of the inductive step for the $L^2$-estimate and we use again the splitting of \eqref{eq:split}. For the term corresponding to the $f_{1,L}$'s we can estimate again recursively 
\[
\begin{split}
 & \sum_{L'\in\Lambda_{\tau-1} ^{|J|^{-1}}}\sum_{ L\in\Lambda_{\tau} ^{|J|^{-1}}(L')} \left\|\Delta_{L^*}f_{1,L}\right\|_{L^2(\R\setminus a^{\tau}J)}^2  \leq C_2(1,a)|a^{\tau-1} J| \sum_{L'\in\Lambda_{\tau-1} ^{|J|^{-1}}} \left\langle |\Delta_{ L'}f|\right\rangle_{1,a^{\tau-1} J} ^2
\\
& \quad\quad \leq C_2(1,a)  C_1(0,\tau-1 ,a^{\tau-1})a^{\tau-1}|J| \left\langle|f|\right\rangle_{B_{(\tau-1)/2,J}} ^2.
  \end{split}
\]
Finally, for the contribution of the $f_{2,L}$'s we use again the $L^2 \to L^2$ boundedness of the smooth Littlewood--Paley square function and the inductive hypothesis to estimate
\[
\begin{split} \sum_{L'\in\Lambda_{\tau-1}^{|J|^{-1}}}\sum_{L\in\Lambda_{\tau} ^{|J|^{-1}}(L')} \left\|\Delta_{L^*}f_{2,L}\right\|_{L^2(\R\setminus a^{\tau}J)}^2 & \lesssim  \sum_{L'\in\Lambda_{\tau-1} ^{|J|^{-1}}} \left\|\Delta_{L'}f\right\|_{L^2(\R\setminus a^{\tau-1} J)} ^2 \\
& \leq   C_2( \tau-1,a^{\tau-1} )|J| \left\langle|f|\right\rangle_{B_{(\tau-1)/2,J}} ^2 .
 \end{split}
\]
We have thus shown that for some numerical constant $c'_{\sigma,\tau}$
\[
C_2(\tau,a^\tau)\leq C_2(1,a)C_1(0,\tau-1,a^{\tau-1})a^{\tau-1} + c'_{\sigma,\tau} C_2(\tau-1,a^{\tau-1}).
\]
This completes the proof of the inductive step and with that the proof of the proposition.
\end{proof}

\begin{remark}\label{rmrk:zygbonami} The first estimate in Proposition~\ref{prop:genzygpos} implies the Zygmund--Bonami inequality of general order. Indeed assume for a moment that $J=[0,1]$ and for $\lambda \in \mathrm{lac}_\tau ^1$ let $L_\lambda\in\Lambda_\tau$ be an interval that has $\lambda$ as an endpoint. We have
\[
\left|\widehat f(\lambda )\right|=\left| \widehat{\Delta_{L_\lambda}(f)}(\lambda)\right|\leq\|\Delta_{L_{\lambda}}(f)\|_{L^1}
\]
for a suitable choice of symbol in the definition of $\Delta_L$ and so the first estimate of the proposition for $\sigma=0$ implies
\[
\left(\sum_{\lambda \in \mathrm{lac}_{\tau} ^1}\left|\widehat{f}(\lambda)\right|^2\right)^{\frac 12}\lesssim \|f\|_{L\log^{\tau/2}L}
\]
which is the Zygmund--Bonami inequality of lacunary order $\tau$. Dualizing and rescaling as in \eqref{eq:rescaling} tells us that for any finite interval $J\subset \R$ with $|J|\in 2^\Z$ we have
\begin{equation}\label{eq:zygbon}
p(y)\coloneqq \sum_{\lambda \in \mathrm{lac}_\tau ^{|J|^{-1}}} a_{\mathbf \lambda} e^{-2\pi i \lambda y},\qquad \|\{a_{\lambda}\}_{\lambda}\|_{\ell^2 }=1 \implies \langle |p|\rangle_{E_{2/\tau},J}\lesssim 1.
\end{equation}
This formulation of the higher order Zygmund--Bonami inequality will be used in several points in the rest of the paper. As it follows from Proposition~\ref{prop:genzygpos} above this makes the proofs in the paper somewhat self contained.
\end{remark}

\begin{remark}\label{rmrk:molecule_pos} One can easily verify that the $L^2$-estimate of Proposition~\ref{prop:genzygpos} can be upgraded to the following form for ``molecules''. Let $\mathcal J$ be a family of pairwise disjoint intervals and $f=\sum_{J\in\mathcal J} b_J$ where $\mathrm{supp} (b_J)\subset J$ for each $J\in\mathcal J$. For every positive integer $\tau$ and $\gamma \geq 2$ there holds
\[
\sum_{L\in \Lambda_\tau} \left\| \sum_{J:\, |J|\geq |L|^{-1}} \Delta_{L} (b_J)\ind_{\R\setminus \gamma J} \right\| _{L^2(\mathbb{R})}^2 \lesssim \sum_{J\in\mathcal{J}} |J| \langle |b_J| \rangle_{B_{(\tau-1)/2},J} ^2.
\]
Indeed an inductive proof is again available. The case $\tau=1$ follows by the same pointwise estimate~\eqref{eq:delta} which implies that
\[
\sum_{J:\, |J|\geq|L|^{-1}}|\Delta_L(b_J)|\ind_{\R\setminus \gamma J}\lesssim  \omega_{|L|^{-1}}* \left( \sum_{|J|\geq |L|^{-1}}\langle |b_J|\rangle_{1,J}\frac{\ind_J}{(|L||J|)^{5}}\right)
\]
which sums using that there are at most two intervals $L\in\Lambda_1$ of any given length.  The inductive step relies again on the identity \eqref{eq:split}, applied to each $b_J$ in place of $f$. Then the contribution of the first term is estimated by an appeal to the case $\tau=1$ followed {by} an application of the first estimate in Proposition~\ref{prop:genzygpos}. The contribution of the second term in \eqref{eq:split} is estimated by the Littlewood--Paley inequalities and the inductive hypothesis. We omit the details.
\end{remark}
We proceed to prove the easier range, corresponding to $|L|<|J|^{-1}$. As in Corollary~\ref{cor:SFCtau} we require cancellation conditions, which in the case at hand amount to vanishing Fourier coefficients of the function at lacunary frequencies corresponding to order $\tau-1$. In this range we can prove the stronger $L^2$ inequality that follows. For simplicity we state the result below for $f$ with support of dyadic length, but it is obvious that this is no real restriction.

\begin{proposition}\label{prop:genzygcanc} Let $J\subset \R$ be a finite interval of dyadic length and $f$ be a compactly supported function with $\mathrm{supp}(f)\subseteq J$. Let $\tau$ be a positive integer. We assume that $\widehat f(\lambda )=0$ for all $\lambda \in\mathrm{lac}_{0} ^{|J|^{-1}}\cup\cdots\cup \mathrm{lac}_{\tau-1} ^{|J|^{-1}}$. Then
\[
 \sum_{\substack{L \in \Lambda_\tau :\\ \, |L|<|J|^{-1}}} \left\| \Delta_{L}f\right\|_{L^2(\R)} ^2 \lesssim |J| \langle |f|\rangle_{B_{(\tau-1)/2},J} ^2.
\]
\end{proposition}

\begin{proof}We argue by induction as in the proof of Proposition~\ref{prop:genzygpos}. Let us denote by $C(\tau)$ the best constant in the inequality we intend to prove. Note that for $\tau=1$ the assumption reads $\int_Jf=0$ and the conclusion $C(1)<+\infty$ follows immediately by the pointwise estimate 
\begin{equation}\label{eq:deltacancel}
|\Delta_L(f)(x)|\lesssim_{M} |L|^2|J|^2 (1+|L||x-c_J|)^{-M}\langle|f|\rangle_{1,J}\lesssim \omega_{|L|^{-1}}*(|L||J|\langle|f|\rangle_{1,J}\ind_J)
\end{equation}
for any large positive integer {$M$}, where $c_J$ is the center of $J$ and $|L||J|<1$. In order to see the first estimate above let us take $\phi_{L}\in\Phi_{L,M}$ to be the symbol of $\Delta_L$ which can be written in the form $\phi_L(x)=e^{i2\pi i c_Lx}|L|\phi(|L|x)$, with $c_L$ denoting the center of $L\in\Lambda_1$. We compute using the cancellation of $f$ and the mean value theorem
\[
\left|\Delta_L(f)(x)\right|\leq \int_J \left|\phi_L(x-y)-\phi_L(x-c_J)\right| |f(y)|\,\d y\lesssim \int_J |J| \sup_{z\in J} |\phi_L ^\prime(x-z)|  |f(y)|\, \d y.
\]
Using that $|c_L|\simeq\mathrm{dist}(L,0)\simeq |L|$ we have for $z\in J$
\[
\begin{split}
|\phi_L ^\prime(x-z)|&\lesssim c_L |L|\phi(|L|(x-z))+|L|^2\phi(|L|(x-z))\lesssim |L|^2  \left(1+|L||x-z|\right)^{-M} 
\\
&\simeq|L|^2  \left(1+|L||x-c_J|\right)^{-M}.
\end{split}
\]
The last approximate equality can be checked by considering the cases $x\in 3J$ and $x\notin 3J$ separately, remembering that $|L||J|<1$. The combination of the last two displays yields the first estimate in \eqref{eq:deltacancel}. The second estimate in \eqref{eq:deltacancel} follows by the first since 
\[
(1+|L||x-c_J|)^{-M}\simeq (1+|L||x-z|)^{-M}
\]
for $z\in J\supseteq\mathrm{supp}(f)$.

 For $\tau>1$ we first do the same reduction as in the proof of Proposition~\ref{prop:genzygpos}. For $\tau>1$ we can estimate
\begin{equation}\label{eq:ngtmr}
\begin{split}
  \sum_{\substack{L\in \Lambda_\tau\\ \, |L|<|J|^{-1}}} \left\| \Delta_{L}f\right\|_{L^2(\R)} ^2 
&\leq \sum_{\substack  {L'\in\Lambda_{\tau-1}\\|L'| <|J|^{-1}}} \sum_{L\in\Lambda_\tau(L')}\left\|\Delta_{L^*}\left(e^{-2\pi i \lambda(L)\bigcdot }\Delta_{L'}f \right)\right\|_{L^2(\R)}^2
\\
 &\qquad+\sum_{L'\in\Lambda_{\tau-1} ^{|J|^{-1}}}\sum_{\substack{L\in\Lambda_\tau(L') \\ |L|<|J|^{-1}} }\left\|\Delta_{L^*}\left(e^{-2\pi i \lambda(L)\bigcdot }f \right)\right\|_{L^2(\R)}^2.
\end{split}
\end{equation}
The first summand above is estimated by $\lesssim C(\tau-1)|J|\langle|f|\rangle_{B_{(\tau-2)/2},J} ^2$ by using the $L^2$-bound for the smooth Littlewood--Paley square function and the inductive hypothesis.

The second summand above can be estimated by
\[
\begin{split}
&\sum_{\ell:\,2^{\ell}< |J|^{-1}}\sum_{L'\in\Lambda_{\tau-1} ^{|J|^{-1}}} \sum_{\substack{L\in\Lambda_\tau(L')\\|L|=2^\ell}} \left\|\Delta_{2^\ell}\left(e^{-2\pi i \lambda(L)\bigcdot }f \right)\right\|_{L^2(\R)}^2 
\\
&\qquad\qquad =
\sum_{\ell:\, 2^{\ell}< |J|^{-1}}
\sum_{\lambda\in\mathrm{lac}_{\tau-1} ^{|J|^{-1}}} 
\sum_{\substack{L\in\Lambda_\tau\\\lambda(L)
=\lambda\\|L|=2^\ell}}
 \left\|\Delta_{2^\ell}\left(e^{-2\pi i \lambda \bigcdot }f \right)\right\|_{L^2(\R)}^2
 \\
&\qquad\qquad  \lesssim \sum_{\ell:\,2^{\ell}< |J|^{-1}}
\sum_{\lambda\in\mathrm{lac}_{\tau-1} ^{|J|^{-1}}} 
 \left\|\Delta_{2^\ell}\left(e^{-2\pi i \lambda \bigcdot }f \right)\right\|_{L^2(\R)}^2
\end{split}
\]
where, in passing to the last line, we used that for each $\lambda\in\mathrm{lac}_{\tau-1}$ there are at most $O_\tau(1)$ intervals $L\in\Lambda_\tau$ of fixed length with $\lambda(L)=\lambda$. Fixing for a moment $2^{\ell}< |J|^{-1} $ and $x\in\R$ we write 
\[ 
\sum_{\lambda\in\mathrm{lac}_{\tau-1} ^{|J|^{-1}}} \left |\Delta_{2^\ell}\left(e^{-2\pi i  \lambda \bigcdot }f \right) (x)\right|^2 =  \left|\Delta_{2^{\ell}}\left(\sum_{\lambda\in\mathrm{lac}_{\tau-1} ^{|J|^{-1}}} a_{\lambda} e^{-2\pi i  \lambda \bigcdot} f\right)(x)\right|^2\eqqcolon |\Delta_{2^\ell}\left(p_{x,\ell}f\right)(x)|^2
\]
where $\{a_\lambda \}_{\lambda}=\{a_\lambda(x,\ell)\}_{\lambda}$ is in the unit ball of $\ell^2 _{\lambda}$ and $p_{x,\ell}$, implicitly defined above, is as in \eqref{eq:rescaling}. Using the cancellation assumptions on $f$ we see that $\int_J p_{x,\ell} f=0$ so by appealing to \eqref{eq:deltacancel} we get
\[
|\Delta_{2^\ell}(fp_{x,k_\tau})|\lesssim \omega_{2^{-\ell}}*\left( 2^\ell|J|\langle|fp_{x,\ell}|\rangle_{1,J}\ind_J\right)\lesssim \omega_{2^{-\ell}}* \left(2^\ell|J|\langle|f|\rangle_{B_{(\tau-1)/2},J}\ind_J\right)
\]
where we used the H\"older inequality in Orlicz spaces together with the Zygmund--Bonami inequality of order $\tau-1$ from Remark~\ref{rmrk:zygbonami} to control $\langle|p_{x,\ell}|\rangle_{E_{2/(\tau-1)}}\lesssim 1$. Squaring the estimate in the last display, integrating, and then summing for $2^\ell<|J|^{-1}$ yields that the second summand in \eqref{eq:ngtmr} is controlled by a constant multiple of $|J|\langle|f|\rangle_{B_{(\tau-1)/2},J} ^2$. We have proved that $C(\tau)\lesssim (1+C(\tau-1))$ and this concludes the proof of the inductive step and of the proposition.
\end{proof}

\begin{remark}\label{rmrk:molecule_canc}As in Remark~\ref{rmrk:molecule_pos} there is an upgrade of the $L^2$-estimate of Proposition~\ref{prop:genzygcanc} from ``atoms'' to ``molecules'' $f=\sum_{J\in\mathcal J}b_J$ where $\mathcal J$ is a family of pairwise disjoint dyadic intervals and each $b_J$ satisfies the cancellation assumptions of Proposition~\ref{prop:genzygcanc}, namely
\[
\sum_{L\in \Lambda_\tau} \left\| \sum_{J:\, |J|< |L|^{-1}} \Delta_{L} (b_J)  \right\| _{L^2(\mathbb{R})}^2 \lesssim \sum_{J\in\mathcal{J}} |J| \langle b_J \rangle_{B_{(\tau-1)/2},J} ^2.
\]
The base case $\tau=1$ is essentially identical to the corresponding step in the proof of Proposition~\ref{prop:genzygcanc} relying on the pointwise estimate
\[
\left| \Delta_L \left(\sum_{J:\, |L|<|J|^{-1}} b_J\right)\right| \lesssim \omega_{|L|^{-1}}*\left(|L|\sum_{J:\, |L|<|J|^{-1}}|J|\langle|b_J|\rangle_{1,J}\ind_J\right).
\]
 This is a consequence of \eqref{eq:deltacancel} using the cancellation assumption $\int b_J=0$ for each $J\in\mathcal J$. For the inductive step with $\tau>1$, denoting again by $C(\tau)$ the best constant in the desired $L^2$-estimate we clearly have that
\[
\begin{split}
\sum_{L\in  \Lambda_\tau} \left\| \sum_{J:\, |J|< |L|^{-1}} \Delta_{L} (b_J)  \right\| _{L^2(\mathbb{R})}^2 &\lesssim 
\sum_{L\in \Lambda_\tau} \left\| \sum_{J:\, |\widehat{L}|^{-1} \leq |J|< |L|^{-1}} \Delta_{|L|}  (e^{-2\pi i \lambda(L)\bigcdot}b_J)  \right\| _{L^2(\mathbb{R})}^2 
\\
&\qquad+ C(\tau-1)
 \sum_{J\in\mathcal{J}} |J| \langle b_J \rangle_{B_{(\tau-2)/2},J} ^2.
 \end{split}
\]
Using a linearization trick as in the proof of Proposition~\ref{prop:genzygcanc} we have 
\[
\begin{split}
& \sum_{L\in \Lambda_\tau} \left| \sum_{J:\, |\widehat{L}|^{-1} \leq |J|< |L|^{-1}} \Delta_{|L|}  (e^{-2\pi i \lambda(L)\bigcdot}b_J) (x) \right|^2
 \\
 &\qquad\qquad =\sum_{\ell\in\Z} \left|\sum_{J:\,   |J|< 2^{-\ell}} \Delta_{2^\ell}\left(\sum_{\substack{L\in\Lambda_{\tau}\\|L|=2^\ell,\,|\widehat{L}|^{-1} \leq |J|}}   a_L e^{-2\pi i \lambda(L)\bigcdot}b_J\right) (x)  \right|^2.
 \end{split}
\]
for some collection $\{a_L\}_{L \in \Lambda_\tau}=\{a_L(x,\ell)\}_{L \in \Lambda_\tau}$ in the unit ball of $\ell^2 _L$. Fixing for the moment $\ell\in\Z$ and $x\in\R$ we have
\[
\begin{split}
p_{x,\ell}&\coloneqq \sum_{\substack{L\in\Lambda_{\tau}\\|L|=2^\ell,\,|\widehat{L}|\geq |J|^{-1 }}}   a_L e^{-2\pi i \lambda(L)\bigcdot}=\sum_{L'\in\Lambda_{\tau-1} ^{|J|^{-1}}}\left(\sum_{\substack{L\in\Lambda_\tau (L')\\|L|2^\ell}}a_L  \right)  e^{-2\pi i \lambda(L)\bigcdot}
\\
&\eqqcolon \sum_{\lambda\in\mathrm{lac}_{\tau-1} ^{|J|^{-1}}}\beta_{\lambda}e^{-2\pi i \lambda \bigcdot}
\end{split}
\]
with $\|\{\beta_{\lambda }\} \|_{\ell^2 _\lambda}=O(1)$. Here we used that there at at most $O(1)$ intervals $L\in\Lambda_\tau$ with fixed length $|L|=2^\ell$ inside $\widehat{L}$. Using the cancellation of $p_{x,\ell} b_J$ we can estimate pointwise $|\Delta_{2^{\ell}}(p_{x,\ell} b_J)|\lesssim 2^{\ell} |J| \langle |p_{x,J} b_J| \rangle_{1,J} \, \omega_{2^\ell} \ast \ind_{J}$, and by Remark~\ref{rmrk:zygbonami} and H\"older's inequality for Orlicz spaces we have that
\[
\langle |p_{x,\ell} b_J|\rangle_{1,J}\lesssim \langle|b_J|\rangle_{B_{(\tau-1)/2},J}.
\]
With this information the proof of the estimate can now be completed summing over $|J|2^{\ell}<1$ as in the proof of Proposition~\ref{prop:genzygcanc}.
\end{remark}

We conclude this section by recording the generalized Zygmund--Bonami inequality under cancellation conditions. This is just a combination of Propositions~\ref{prop:genzygpos} and~\ref{prop:genzygcanc}. 

\begin{cor}\label{cor:genzygbonami}Let $\sigma$ be a nonnegative real number
and $\tau$ be a positive integer. Assume that $\mathrm{supp}(f)\subset J$ for some finite interval $J$ and that $\widehat f(\lambda )=0$ for all $\lambda \in\mathrm{lac}_{0} ^{|J|^{-1}}\cup\cdots\cup \mathrm{lac}_{\tau-1} ^{|J|^{-1}}$. Then
\[
\left( \sum_{L\in \Lambda_\tau }  \left\langle|\Delta_{L} (f) |\right 
\rangle_{B_{\sigma/2},\gamma J} ^2 \right)^{1/2}
\lesssim \l |f|\r _{B_{(\sigma+\tau)/2}, J}.
\]
\end{cor}

\section{An $L^{B_{\sigma/2}}(\mathbb{R})$ Calderon-Zygmund decomposition}\label{sec:CZdecomp} We describe in this section a Calder\'on-Zygmund decomposition adapted to the (global) Orlicz space $L^{B_{\sigma/2}}(\R)$ for $\sigma\geq 0$. Such a Calder\'on-Zygmund decomposition, which is influenced by the one appearing in \cite[{Appendix A}]{UMP}, is available to us because of the specific choice of the Young function $B_{\sigma/2}$ and it is adapted to the finite order lacunary setup.

Recall that for $\sigma\geq 0$ we write $f\in  L^{B_{\sigma/2}}(\R)$ if for some (or equivalently all) $\lambda>0$ there holds
\[
\int_{\R} B_{\sigma/2}\left(\frac{|f(x)|}{\lambda}\right)\, \d x<+\infty.
\]
There is an Orlicz maximal operator associated with $B_\sigma$
\[
\M_{B_{\sigma/2}} f(x) \coloneqq \sup_{Q \ni x} \langle |f| \rangle_{B_{\sigma/2},Q},\qquad x\in\R,
\]
with the supremum being over all intervals $Q$ of $\R$ containing $x$. The dyadic version of $\M_{B_{\sigma/2}}$ is defined similarly with the supremum over all dyadic intervals $Q\in\mathcal D$ with $\mathcal D$ some dyadic grid. We will write $\M_{B_{\sigma/2},\mathcal D}$ for the dyadic version. Below we denote by $Q\in\mathcal D$ a dyadic interval, $Q^{(1)}$ its dyadic parent and set $Q^{(k+1)}$ to be the dyadic parent of $Q^{(k)}$. 

\begin{remark}[Existence of stopping intervals]\label{rmrk:stopping} For the Calder\'on-Zygmund decomposition we will choose stopping intervals that are maximal under the condition $\langle|f|\rangle_{B_{\sigma/2},I}>\lambda$. The existence of these stopping intervals relies on the following fact: If  $f\in L^{B_{\sigma/2}}(\R)$ for some $\sigma\geq 0$ and $I$ is a dyadic interval in some grid $\mathcal D$, then $\langle |f| \rangle_{B_{\sigma/2},I^{(k)}} \to 0$ as $k \to +\infty$.  This can be easily proved using for example the fact that the Young function $B_{\sigma/2}$ is submultiplicative.
\end{remark}

\begin{proposition}\label{prop:CZ_decomp} Let $\sigma$ be a fixed nonnegative integer, $f\in L^{B_{\sigma/2}}(\R)$,  and  $\alpha>0$. There exists a collection $\mathcal J$ of pairwise disjoint dyadic intervals $J$ and a decomposition of $f$
\[
f=g+ b_{\mathrm{canc},\sigma} + b_{\mathrm{lac}, \sigma}
\]
such that the following hold:
\begin{itemize}[leftmargin=2em,itemsep=1em]
\item[(i)] The function $g$ satisfies $\|g\|_{L^\infty(\R)} \lesssim \alpha$ and  $\|g\|_{L ^1(\R)} \lesssim  \|f\|_{L ^1(\R)}$.
\item[(ii)] The function $b_{\mathrm{canc},\sigma}$ is supported in $\cup_{J\in\mathcal J}J$ and in particular
\[
b_{\mathrm{canc},\sigma}=\sum_{J\in\mathcal J} b_J,\quad  \supp(b_J)\subseteq J,\qquad \widehat{b_J}(\lambda) =0\qquad \forall \lambda \in \mathrm{lac}_0\cup\cdots\cup \mathrm{lac}_{\sigma}.
\]
Furthermore we have that $\langle|b_J|\rangle_{B_{\sigma/2},J}\lesssim \alpha$ for all $J\in\mathcal J$ 
and
\[
\sum_{J\in\mathcal J}|J| \leq \int_{\R} B_{\sigma/2}\left(\frac{|f|}{\alpha}\right).
\]
\item[(iii)] The function $b_{\mathrm{lac},\sigma}$ is also supported on $\cup_{J\in\mathcal J} J$ and satisfies
\[
\|b_{\mathrm{lac},\sigma}\|_{L^2(\R)} ^2 \lesssim  \sum_{ J \in\mathcal{J}}|J| \langle |b_J|\rangle_{B_{\sigma /2},J} ^2
\lesssim
\alpha^2\int_{\R} B_{\sigma/2}\left( \frac{|f|}{\alpha} \right).
\]
\end{itemize}
\end{proposition}

\begin{proof} We begin by recalling that $f\in L^{B_{\sigma/2}} (\R)$ implies that $\int_{\R} B_{\sigma/2}(|f|/\alpha)<+\infty$ for all $\alpha>0$. By Remark~\ref{rmrk:stopping} and \cite[{Theorem 5.5}]{UMP} we have that the dyadic Orlicz maximal operator $\M_{B_{\sigma/2},\mathcal D}$ satisfies
\[
|E_\alpha|\coloneqq |\{x\in\R:\, \M_{B_{\sigma/2},\mathcal D}f(x)>\alpha\}|\leq \int_{\R} B_{\sigma/2}\left(\frac{|f|}{\alpha}\right),\qquad \alpha>0.
\]
Letting $\mathcal J$ denote the collection of maximal dyadic intervals contained in $E_\alpha$ we have that for every $J\in \mathcal J$
\[
\alpha<\langle |f| \rangle_{B_{\sigma/2},J} \leq 2 \alpha,\qquad  \sum_{J\in\mathcal J}|J| \leq \int_\R B_{\tau/2}\left(\frac{|f|}{\alpha}\right);
\]
The upper bound in the approximate inequality of the leftmost estimate above follows by the maximality of $J$ and the convexity of the Young function of $B_{\sigma/2}$ which implies that
\[
\langle |f|\rangle_{B_{\sigma/2}, J}\leq \rho \langle |f|\rangle_{B_{\sigma/2}, \rho J},\qquad \rho>1;
\] 
see\cite[{Proposition A.1}]{UMP} and  \cite[{eq. (5.2)}]{UMP}.  One routinely checks that $g\coloneqq f\ind_{\R\setminus \cup_{J\in\mathcal J} J}$ satisfies (i).

For the ``atoms'' we set $f_J\coloneqq f\ind_J$ and define
\[
b_{J,\mathrm{lac},\sigma}(y)\coloneqq \sum_{\rho=0} ^{\sigma} \left(\sum_{\lambda \in\mathrm{lac}_\rho ^{|J|^{-1}}}  \widehat{f_J}(\lambda )e^{2\pi i \lambda y}\right)\frac{\ind_J(y)}{|J|},\qquad b_J\coloneqq f_J -b_{J,\mathrm{lac},\sigma},
\]
and $b_{\mathrm{lac},\sigma}\coloneqq \sum_{J\in\mathcal J} b_{J,\mathrm{lac},\sigma}$ and $b_{\mathrm{canc},\sigma}\coloneqq \sum_{J\in\mathcal J} b_J$. The cancellation conditions of (ii) for $b_{\mathrm{canc},\sigma}$ follow immediately by the definition above. Furthermore by the H{\" o}lder inequality for Orlicz spaces and the Zygmund--Bonami inequality of order $\rho\in\{1,\ldots,\sigma\}$  as in Remark~\ref{rmrk:zygbonami}, one sees that
\[
\langle |b_{J,\mathrm{lac}, \sigma}|\rangle_{B_{\tau/2},J}\lesssim \langle |b_{J,\mathrm{lac}, \sigma}|\rangle_{2,J} \lesssim  \langle |f_J|\rangle_{B_{\sigma/2},J}\lesssim \alpha.
\]
This and the triangle inequality also yield $\langle|b_J|\rangle_{B_{\sigma/2},J} \lesssim \langle |f_J|\rangle_{B_{\sigma/2},J}\lesssim \alpha$ thus completing the proof of the desired conclusions in (ii). Finally for (iii) we estimate as above
\[
\|b_{\mathrm{lac},\sigma}\|_{L^2(\R)} ^2 \lesssim  \sum_{J\in\mathcal J}|J|\langle |f_J|\rangle_{B_{\sigma/2},J} ^2 \lesssim \alpha^2 \int_\R B_{\tau/2}\left(\frac{|f|}{\alpha}\right)
\]
and the proof is complete.
\end{proof}

\section{Proof of Theorem \ref{thm:main} and Corollaries}\label{sec:proofmain} In the first part of this section we compile together the results of the previous sections to conclude the proof of Theorem \ref{thm:main}.  In the second part we show how to conclude our corollaries, namely Theorem~\ref{thm:Marcink} and~\ref{thm:horm}.

\subsection{Proof of Theorem \ref{thm:main}} Let us fix a positive integer $\tau$ and $m\in R_{2,\tau}$. Before entering the heart of the proof we note that it suffices to prove the theorem for multipliers $m$ having the form
\[
m=\sum_{I\in\mathcal I} c_I \ind_I
\]
where the family of intervals $\mathcal I$ has overlap at most $N$, for each $I\in\mathcal I$ there exists a unique $L=L_I\in \Lambda_\tau$ such that $I\subset L_{I}$ and  for each fixed $L\in\Lambda_\tau$ there holds
\[
\sum_{I:\, L_I=L}|c_I|^2 \leq N^{-1}.
\]
See the analysis in \cite[{p. 533}]{TW} for the details of this approximation argument. For $m$ of this form, we now can write
\[
\mathrm{T}_m(f)=\sum_{I\in\mathcal I} c_I\mathrm{P}_If=\sum_{L\in\Lambda_\tau}\sum_{I:\, L_I=L} c_I\mathrm{P}_I(\Delta_{L}f),\qquad \mathrm{P}_I f\coloneqq (\ind_I \hat f)^\vee,
\]
a fact that we will use repeatedly in what follows.

\subsubsection{The upper bound in Theorem~\ref{thm:main}}Let $f$ be a function in $L^{B_\tau}(\mathbb{R})$ and $\alpha>0$ be fixed. We decompose $f$ according to the Calder\'on-Zygmund decomposition in Proposition~\ref{prop:CZ_decomp} with $\sigma=\tau$ yielding
\[
f=g+b_{\mathrm{canc},\tau}+b_{\mathrm{lac},\tau}.
\]
We directly estimate $g+b_{\mathrm{lac},\tau}$ in $L^2$ using (i) and (iii) of Proposition~\ref{prop:CZ_decomp}
\[
 |\left\{ x\in\R: \, |\mathrm{T}_m(g+b_{\mathrm{lac},\tau})(x)|>\alpha \right\}| \lesssim \frac{1}{\alpha^2} \left\|g+b_{\mathrm{lac},\tau}\right\|_{L^2(\R)} ^2 \lesssim \int_\R B_{\tau/2}\left(\frac{|f|}{\alpha}\right).
\]
The main part of the proof deals with the bad part $b_{\mathrm{canc},\tau}=\sum_{J\in\mathcal J } b_J$ and it suffices to estimate
\[
|\{x\in\R\setminus \cup_{J\in\mathcal J} 6J:\, |\mathrm{T}_m(b_{\mathrm{canc},\tau})|>\alpha\}|
\]
as the measure $|\cup_{J\in\mathcal J} 6J|$ satisfies the desired estimate by (ii) of Proposition~\ref{prop:CZ_decomp}. We will adopt the splitting
\[
\begin{split}
\mathrm{T}_m\left(\sum_J b_J\right)&=\sum_I c_I\mathrm{P}_I \left(\sum_{J:\, |J|\geq  |L_{I}|^{-1}}\Delta_{L_I}( b_J)\ind_{\R\setminus 3J}\right)+\sum_I c_I\mathrm{P}_I \left(\sum_{J:\, |J|<  |L_{I}|^{-1}}\Delta_{L_I}( b_J)\right)
\\
&\qquad +\sum_I c_I\mathrm{P}_I \left(\sum_{J:\, |J|\geq |L_{I}|^{-1}}\Delta_{L_I}( b_J)\ind_{3J}\right)\eqqcolon \mathrm{I}+\mathrm{II}+\mathrm{III}.
\end{split}
\]
The main term is $\mathrm{III}$. Indeed we can estimate the term $\mathrm{I}$ in $L^2(\R)$ using Remark~\ref{rmrk:molecule_pos}, while $\mathrm{II}$ is also estimated in $L^2(\R)$ using Remark~\ref{rmrk:molecule_canc} this time. Note that each $b_J$ has the required cancellation by (ii) of Proposition~\ref{prop:CZ_decomp}. Using also (ii) of Proposition~\ref{prop:CZ_decomp} to control the averages $\langle|b_J|\rangle_{B_{(\tau-1)/2},J} \lesssim \langle|b_J|\rangle_{B_{\tau/2},J}\lesssim\alpha $  we have
\[
\left|\left\{|\mathrm{I}+\mathrm{II}|>\alpha\right\}\right| \lesssim\frac{1}{\alpha^2}\sum_{J\in\mathcal J}|J|\langle|b_J|\rangle_{B_{(\tau-1)/2},J} ^2\lesssim \sum_{J\in\mathcal J}|J|\lesssim \int_{\R}B_{\tau}(|f|/\alpha)
\]
as desired. 

It remains to deal with  $\mathrm{III}$ and we make a further splitting. Let $k_I\in\Z$ be such that $2^{k_I}<|I|\leq 2^{k_I+1}$. Of course we will always have that $|L_I|\geq 2^k_I$ since $I\subseteq L_{I}$. We write
\begin{align*}
\mathrm{III} & =\sum_I c_I\mathrm{P}_I \left(\sum_{J:\, 2^{-k_I}>|J|\geq  |L_{I}|^{-1}}\Delta_{L_I}( b_J)\ind_{3J}\right)+\sum_I c_I\mathrm{P}_I \left(\sum_{J:\, |J|\geq  2^{-k_I}}\Delta_{L_I}( b_J)\ind_{3J}\right) \\
& \eqqcolon \mathrm{III_1}+\mathrm{III_2}.
\end{align*}
We first handle the term $\mathrm{III}_1$. Let $\Delta_I$ be the smooth frequency projections on the interval $I$ as fixed in \S\ref{sec:lacnot}; then in particular we can write $\mathrm{P}_I \Delta_I=\mathrm{P}_I$ and we have the familiar pointwise estimate
\[  
|\Delta_I( \Delta_{L_I}(b_J) \ind_{3J}) | \lesssim \omega_{|I|^{-1}} *\left( \langle |\Delta_{L_I}(b_J)|\rangle_{1,3J}\ind_J\right) 
\]
as $|I|^{-1}\simeq 2^{-k_I}>|J|$. We thus get
\[
\begin{split}
|\{|\mathrm{III}_1|>\alpha \}|&  \lesssim\frac{1}{\alpha^2}\sum_{L\in\Lambda_\tau}N\sum_{I:\,  L_I=L}|c_I|^2 \int_{\R}\left|\omega_{|I|^{-1}}*\left(\sum_{J:\, |J|\geq 2^{-k_I}} \langle |\Delta_L(b_J)|\rangle_{1,3J} \ind_J \right)\right|^2
\\
&\lesssim \frac{1}{\alpha^2}\sum_{J\in\mathcal J} \sum_{L\in\Lambda_\tau ^{|J|^{-1}}}  |J|\langle |\Delta_L(b_J)|\rangle_{1,3J}^2 \lesssim \sum_{J\in\mathcal J}|J|\lesssim\int_\R B_{\tau/2}\left(\frac{|f|}{\alpha}\right)
\end{split}
\]
where we used the $\ell^2$-control on the coefficients $\{c_I\}_{L_I=L}$ in passing to the second line and the generalized Zygmund--Bonami inequality of Proposition~\ref{prop:genzygpos} together with the properties of the Calder\'on-Zygmund decomposition in the penultimate approximate inequality.

The steps required for dealing with the the term $\mathrm{III}_2$ are essentially the same as those in \cite{TW}, however, as here we are dealing with a higher order set up, we include them for the sake of completeness. We will split the estimate for $\mathrm{III}_2$ into two parts. In the first we keep the part of the multiplier $\ind_I=\ind_{[\ell_I,r_I]}$ at scale $O(|J|^{-1})$ around its singularities which are at the endpoints. We make this precise now.

Let $0\leq \psi_{I,J}\leq 1$ be a smooth bump which is $1$ on the $(10|J|)^{-1}$-neighborhood of the endpoints $\{\ell_I,r_I\}$ of $I$ and vanishes off the $(5|J|)^{-1}$-neighborhood of the endpoints, and satisfies $\|\partial^\alpha\psi_{I,J}\|_{L^\infty}\lesssim |J|^\alpha$ for all $\alpha$ up to some sufficiently large integer $M$. Letting $\Psi_{I,J}$ denote the operator with symbol $\psi_{I,J}$ we define 
\[
{\mathcal E}(\{b_J\}_{J\in \mathcal J})\coloneqq \sum_I c_I \mathrm{P}_I \left(\sum_{J:\, |J|\geq 2^{-k_I}} \Psi_{I,J} \left(\Delta_{L_I} (b_J) \ind_{3J}\right)\right) 
\]
The following lemma shows that the operator ${\mathcal E}(\{b_J\}_{J\in \mathcal J})$ can be dealt with, again, by $L^2$-estimates.

\begin{lemma}\label{lem:local_sing} We have the estimate
\[
\left\|{\mathcal E}(\{b_J\}_{J\in\mathcal J}) \right\|_{L^2(\R)} ^2 \lesssim \sum_{J\in\mathcal J} |J|  \langle|b_J|\rangle_{B_{\tau/2},J} ^2\lesssim \alpha^2 \int_{\R} B_{\tau/2}\left(\frac{|f|}{\alpha}\right).
\]
\end{lemma}

\begin{proof} First note that by the overlap assumption on the intervals $I$ we have
\[
\left\|{\mathcal E}(\{b_J\}_{J\in\mathcal J}) \right\|_{L^2(\R)} ^2 \lesssim N  \sum_{L\in\Lambda_\tau}\sum_{I:\,L_I=L}  |c_I|^2 \int_{\R} \left( \sum_{J:\, |J|\geq 2^{-k_I}} \left|\Psi_{I,J} (\Delta_{L} (b_J) \ind_{3J})\right|\right)^2.
\]
The following pointwise estimate can be routinely verified
\[
|\Psi_{I,J}(\Delta_{L}(b_J)\ind_{3J})(x)|\lesssim  \M(\ind_J)(x)^{10} \langle|\Delta_{L}( b_J)|\rangle_{1,3J}.
\]
Using the Fefferman-Stein inequality and rearranging the sums we can conclude that
\[
\begin{split}
 \left\|{\mathcal E}(\{b_J\}_{J\in\mathcal J}) \right\|_{L^2(\R)} ^2 & \lesssim N \sum_{k} \sum_{L\in\Lambda_\tau ^{2^k}}\,\sum_{\substack{I:\,L_I=L \\ \,\,\,\,\,\ell_I=\ell}} |c_I|^2 \sum_{|J|\geq 2^{-k}} \langle|\Delta_{L }( b_J)|\rangle_{1,3J} ^2 |J|
\\
&\leq  \sum_{J\in\mathcal J}|J| \sum_{L\in\Lambda_\tau  ^{|J|^{-1}}}  \langle|\Delta_{L } (b_J)|\rangle_{1,3J}^2
\end{split}
\]
where we used the $\ell^2$-control on the coefficients $\{c_I\}_{L_I=L}$  in passing to the second line. An appeal to the generalized Zygmund--Bonami inequality of order $\tau$ in Proposition~\ref{prop:genzygpos}  concludes the proof of the lemma.
\end{proof}

We are left with studying the contribution of the operator
\[
{\mathcal L} (\{b_J\}_{J\in\mathcal J}) \coloneqq \sum_I c_I \mathrm{P}_I \sum_{|J|\geq 2^{-k_I}} (\mathrm {Id}-\Psi_{I,J}) \left( \Delta_{L_I}(b_J)\ind_{3J}\right).
\]
For this we consider the multiplier $\zeta_{I,J}\coloneqq \ind_I (1-\psi_{I,J})$
which is a smooth function with values in $[0,1]$, supported in $I$, is identically $1$ on $|x-c_I|\lesssim |I|$ and drops to $0$ with derivative $O(|J|)$ close to the endpoints of $I$. More generally, one easily checks that $\zeta_{I,J}$ satisfies
\[
|\partial^\alpha \zeta_{I,J}| \lesssim |J|^\alpha \ind_{I_{\mathrm{left}}(J)\cup I_{\mathrm{right}}(J)}\qquad \forall \alpha\geq 1,
\]
where
\[
I_{\text{left}}(J)\coloneqq \left[\ell_I+10^{-1}|J|^{-1},\ell_I+5^{-1}|J|^{-1}\right]\subset I   
\]
and
\[
I_{\text{right}}(J)\coloneqq \left[r_I-5^{-1}|J|^{-1},r_I-10^{-1}|J|^{-1}\right]\subset I.
\]
Remembering that we are dealing with the case $|I||J|\gtrsim 1$ we see that the function $\zeta_{I,J}$ has support of size $O(|I|)$ and $\alpha$-derivatives of size $O(|J|^\alpha)$; thus the function $ \zeta_{I,J}$ is not a good kernel. The important observation is however that the derivatives of $\widecheck{\zeta_{I,J}}$ of order $\alpha \geq 1$ have support of size $|I_{\mathrm{left}}(J)\cup I_{\mathrm{right}}(J)|\simeq |J|^{-1}$.

Given an interval $J\subset \R$ we will also use an auxiliary function $\rho_J$ defined as follows. We choose $0\leq \rho\leq 1$ to be a smooth bump function which is identically $1$ on $[-1,1]$ and vanishes off $[-3/2,3/2]$ and define $\rho_J(x)\coloneqq \rho\left({x}/{|J|}\right)$ for $x\in\R$.

\begin{lemma}\label{lem:ngtmr} Let $I,J$ be intervals and $\zeta_{I,J}$ and $\rho_J$ be defined as above.

 If $\widecheck{\mu_{I,J}} \coloneqq (1-\rho_J)  \widecheck{\zeta_{I,J}}$ then for any nonnegative integers $\gamma,\beta$  there holds
\[
\left|\partial_\xi  ^\beta  \mu_{I,J}(\xi)\right|\lesssim |J|^{\beta}  \left(1+|J|\dist(\xi,\R\setminus I)\right)^{-\gamma} , \qquad \xi\in\R.
\]
\end{lemma}

\begin{proof} We begin by noting that since $\zeta_{I,J}$ is a Schwartz function and $(1-\rho_J)$ is a smooth bounded function, we have  we have that $(1-\rho_J) \widecheck{\zeta_{I,J}}$ is a Schwartz function. Furthermore, by the comments preceding the statement of the lemma we have that $\zeta_{I,J}$ satisfies
\[
|\partial_\xi ^\alpha \zeta_{I,J}| \lesssim |J|^\alpha \ind_{I_{\mathrm{left}}(J)\cup I_{\mathrm{right}}(J)}\qquad \forall \alpha\geq 1.
\]
Note that by symmetry it suffices to prove the estimate for $\xi\in \R$ such that $\dist(\xi,\R\setminus I)=|\xi- \ell_I|$ where we remember that $I=[\ell_I,r_I]$. For simplicity we will write  $I(J)$ for $I_{\mathrm{left}}(J)$. Thus the conclusion of the lemma reduces to showing
\[
|\partial_\xi ^\beta   \mu_{I,J} (\xi )|\lesssim_\beta |J|^\beta  (1+|J||\xi-\ell_I|)^{-\gamma}.
\]
We will henceforth drop the subindices $I,J$ in order to simplify the notation. We record the following standard integration by parts identity; for nonnegative integers $\gamma,\nu$ we have
\[
\begin{split}
(\partial_x-2\pi i \ell_I) ^\gamma  [ \widecheck{\zeta}  (x) ] 
= (-1)^\nu\frac{(2\pi i )^{\gamma-\nu}}{x^\nu} \int_{\R} \partial_\xi  ^{\nu} \left[(\xi-\ell_I)^\gamma \zeta (\xi)\right] e^{2\pi i x \xi}\, \d \xi.
\end{split}
\]
In order to make sure that all terms in $\partial_\xi  ^{\nu} \left[(\xi-\ell_I)^\gamma \zeta (\xi)\right]$ contain at least one derivative we take $\nu > \gamma$. Then we have
\[
\begin{split}
\left|\partial_\xi  ^{\nu} \left[ (\xi-\ell_I)^\gamma \zeta (\xi)\right] \right| 
\lesssim \sum_{k=0} ^\gamma\left |\partial_\xi ^k (\xi-\ell_I)^\gamma \partial_\xi ^{\nu-k}[\zeta (\xi)]\right | 
& \lesssim \sum_{k=0} ^\gamma |\xi-\ell_I|^{\gamma-k} |J|^{\nu-k} \ind_{I(J)}(\xi)
\\
& \lesssim |J|^{\nu-\gamma}\ind_{I(J)}(\xi)
\end{split}
\]
provided that $\nu>\gamma$. Plugging this estimate into our integration by parts identity we get
\begin{equation}\label{eq:lambdader}
\left| (\partial_x-2\pi i \ell_I) ^\gamma   [ \widecheck{\zeta}  (x)]\right|\lesssim \frac{|J|^{\nu-\gamma-1}}{|x|^{\nu}},\qquad \nu>\gamma.
\end{equation}
Using this estimate we have for nonnegative integers $\beta,\gamma$
\[
\begin{split}
 \partial_\xi  ^\beta [\mu(\xi)]&=\frac{ (-2\pi i)^\beta  }{(2\pi i (\xi-\ell_I))^\gamma}\int_{\R}  (\partial_x-2\pi i \ell_I) ^\gamma \left[x ^\beta (1-\rho (x/|J|)) \widecheck{\zeta}  (x) \right] e^{-2\pi i x \xi}  \, d x.
\end{split}
\]
Using \eqref{eq:lambdader} with $\nu$ large together with the fact that $\supp(1-\rho_J)\subset \{|x|\gtrsim |J|\}$ and that $\supp (\partial_x[\rho_J]) \subset \{|x|\simeq |J|\}$ and combining with the previous identity yields
\[
\begin{split}
&\left|\partial_\xi ^\beta   [\mu(\xi)] \right|\lesssim \frac{1}{|\xi-\ell_I|^\gamma} \int_{\R} \sum_{\substack{k_1+k_2+k_3=\gamma\\k_1\leq \beta}} |x|^{\beta-k_1} \left|\partial_x ^{k_2} (1-\rho (x/|J|))(\partial_x -2\pi i \ell_I)^{k_3} [\widecheck{\zeta}(x)]\right| \, \d x
\\
&\quad \leq  \sum_{\substack{k_1+k_3=\gamma\\k_1\leq \beta }} \frac{|J|^{\nu-k_3-1}}{|\xi-\ell_I|^\gamma}  \int_{|x|\gtrsim |J|} |x|^{\beta-k_1-\nu}  \, \d x
\\
&\qquad\qquad+ \sum_{\substack{k_1+k_2+k_3=\gamma\\k_1\leq \beta,\, k_2\geq 1 }}\frac{|J|^{\nu-k_3-1}}{|\xi-\ell_I|^\gamma} \int_{|x|\simeq |J|}  |x|^{\beta-k_1-\nu} |J|^{-k_2}  \, \d x
\\
&\quad \lesssim \frac{|J|^{\beta-\gamma}}{|\xi-\xi_J|^\gamma}.
\end{split}
\]
Combining this estimate for general $\gamma$ with the special case $\gamma=0$ yields the conclusion of the lemma.
\end{proof}

We can now prove the desired estimate for the remaining term.

\begin{lemma}\label{lem:loc_middle}There holds
\[
\int_{\R\setminus \cup_{J\in\mathcal J} 6J} |{\mathcal L} (\{b_J\}_{J\in\mathcal J})| \lesssim \sum_{J\in\mathcal J}|J| \langle |b_J|\rangle_{B_{\tau/2},J}  .
\]
\end{lemma}

\begin{proof} For convenience we set
\[
{\mathrm L}_{I,J}\coloneqq  \mathrm{P}_I({\mathrm Id}-\Psi_{I,J}),\quad F_{I,J}\coloneqq \Delta_{L_I}(b_J)\ind_{3J},\,\,\,\,\,\,\, {\mathcal L} (\{b_J\}_{J\in\mathcal J}) = \sum_I\sum_{J:\, |J|\geq 2^{-k_I}}  c_I {\mathrm L}_{I,J}( F_{I,J}) .
\]
We immediately note that it will be enough to prove the desired estimate for a single $b_J$ and then sum the estimates. Furthermore, by translation and scale invariance it will be enough to to assume that $J=[-|J|/2,|J|/2]$; here we critically use that the operator $L_{I,J}$ depends only on the length and not on the position of $J$. The left hand side in the conclusion of the lemma for a single such $b_J$ can be estimated by
\[
\begin{split}
A_J&\coloneqq \int_{|x|\geq 3|J|}\left | \sum_{I:\, 2^{k_I}\geq |J|^{-1}} c_I {\mathrm L}_{I,J}F_{I,J}\right|  = \int_{|x|\geq 3|J|} \left|\sum_{ I:\, 2^{k_I}\geq |J|^{-1}} c_I (\widecheck{\zeta_{I,J}}*F_{I,J})(x) \right| \, \d x
\\
 & \lesssim |J|^{-1/2}\left(\int_{|x|>3|J|}|x|^2 \left| \sum_{I:\,2^{k_I}\geq 1} c_I(\widecheck{\zeta_{I,J}}* {F_{I,J}} ) (x)\right|^2\, \d x\right)^{\frac12}.
  \end{split}
\]
 Now let $\rho$ be as before. It is then the case that for $|x|\geq 3$ and $|y|\leq 3/2$ we have 
\[
|x-y|\geq \frac{1}{2}|x| \geq \frac{3}{2},\qquad 1-\rho\left( \frac{x-y}{|J|} \right) = 1 - \rho_J(x-y)=1
\]
for such pairs $(x,y)$. As $F_{I,J}$ is supported in $[-3|J|/2,3|J|/2]$ we have for all $|x|\geq 3|J|$ that
\[
\begin{split}
(\widecheck{\zeta_{I,J}}*F_{I,J})(x)=\int_{[-3|J|/2,\,3|J|/2]} F_{I,J}(y)\widecheck{\zeta_{I,J}}(x-y)\left(1-\rho_J(x-y))\right)\, \d y.
\end{split}
\]
Using this identity and setting $\widecheck{\mu_{I,J}}\coloneqq (1-\rho_J)\widecheck{\zeta_{I,J}}$ we get
\[
\begin{split}
A_J&\lesssim \left(\int_{\R}  \left| \sum_{I:\, 2^{k_I}\geq |J|^{-1}} c_I   \partial_{\xi}\left[ \widehat{F_{I,J}} (\xi) \mu_{I,J}(\xi)\right]  \right|^2\, \d \xi\right)^{\frac12}.
\end{split}
\]
Using the elementary estimates $\|\widehat{F_{I,J}}\|_{L^\infty(\R)}\leq \|F_{I,J}\|_{L^1(\R)}$ and $\|\partial_\xi \widehat{F_{I,J}}\|_{L^\infty(\R)}\leq|J| \|F_{I,J}\|_{L^1(\R)}$, together with the estimate of Lemma~\ref{lem:ngtmr} for $\beta\in\{0,1\}$ we get for $\gamma$ a large positive integer of our choice
\[
\partial_\xi ^\beta [\mu_{I,J}(\xi)]\lesssim \frac{|J|^\beta}{(1+|J|\dist(\xi,\R\setminus I) )^\gamma}\lesssim |J|^\beta \M(\ind_{I_{\mathrm{left}}(J)\cup I_{\mathrm{right}}(J)})^\gamma .
\]
Hence, by using the Fefferman-Stein inequality, the Cauchy-Schwarz inequality, and the $N$-overlap assumption on the intervals $I$, we get
\[
A_J  \lesssim   |J|^{-1/2} \left(\sum_{2^{k_I}\geq |J|^{-1}} |c_I|^2 |J|^4 \langle|\Delta_{L_I}(b_J)|\rangle_{1,3J} ^2    N \left|I_{\mathrm{left}}(J)\cup I_{\mathrm{right}}(J)\right|\right)^{\frac12}
\]
Here note that we use that  
$I_{\mathrm{left}}(J) \cup I_{\mathrm {right}} (J) \subsetneq I$ by construction, and hence 
\[
\sum_{2^k_I\geq |J|^{-1}}  \ind_{I_{\mathrm{left}}(J)\cup I_{\mathrm{right}}(J)}  \leq \sum_{I} \ind_I \leq N .
\]
 Further, using also the control on the $\ell^2$-norm of the sequence $\{c_I\}_I$ yields
\[
\begin{split}
A_J &\lesssim   |J|^{-1/2}\left(\sum_{L\in \Lambda_\tau ^{|J|^{-1}}} |J|^3 \langle|\Delta_{L} (b_J)|\rangle_{1,3J} ^2 \right)^{1/2}= |J| \left(\sum_{L\in\Lambda_\tau ^{|J|^{-1}} } \langle |\Delta_{L} (b_J)|\rangle_{1, 3J } ^2\right)^{1/2}
\\
&\lesssim |J|\langle|b_J|\rangle_{B_{\tau/2},J}
\end{split}
\]
by the generalized Zygmund--Bonami inequality of Proposition~\ref{prop:genzygpos}. This concludes the proof of the lemma.
\end{proof}

Using Lemmas~\ref{lem:local_sing} and~\ref{lem:loc_middle} we complete the estimate for the term $\mathrm{III}$ and with that the proof of the endpoint bound of the theorem. 

\subsubsection{Optimality in Theorem~\ref{thm:main}}\label{sec:optr2tau} We briefly comment on the optimality of the Young function $t\mapsto t(\log(e+t))^{\tau/2}$ in the upper bound of the theorem. Suppose that  $r>0$ is such that whenever  $\mathrm T_m$ is an $R_{2,\tau}$ multiplier operator then the bound of Theorem~\ref{thm:main} holds with $r$ in the place of $\tau$. Since $\mathrm T_m$ is $L^2$-bounded, it follows by a Marcinkiewicz interpolation type of argument that the $L^p(\R)$-bounds for the Littlewood--Paley square function $\mathrm {LP}_\tau$ of order $\tau$ can be estimated by
\[ 
\|\mathrm{LP}_\tau \|_{p\to p} \lesssim \left(\mathbf E\left\| \sum_{L\in\Lambda_\tau} \eps_{L} \mathrm P_{L}\right\|_{p\to p}^p\right)^{1/p} \leq \sup_{\|m\|_{R_{2,\tau}}=1} \|\mathrm T_m \|_{L^p\to L^p } \lesssim (p-1)^{-(r+1)} 
\]
as $p \rightarrow 1^+$, where the expectation in the display above is over independent choices of random signs $\{\eps_{L}\}_{L}$. However, a modification of an example in \cite{Bou}, see \cite[{\S 3}]{Bakas_L^p}, shows that the estimate in the display above does not hold for $r< \tau/2 $. This argument also shows that our theorem implies that the $L^p$ bounds for $R_{2,\tau}$ multipliers are $O( \max(p,p')^{1+\tau/2})$.

Alternatively, sharpness can also be obtained by adapting the corresponding argument in \cite[{\S3.2}]{TW} to the higher order case. Let us briefly outline the second order case. For a smooth function $\psi$ supported in $[-1/2,1/2]$ with $\psi (0) = 1$ and $(k,l)\in\Z^2$ with $k>l$ we consider the multiplier $m_{(k,l)}$ given by
\[
m_{(k,l)} (\xi ) \coloneq m_0 \left(  \frac{\xi - 2^k}{2^{l-1}} \right), \quad \text{where} \quad m_0 (\xi) \coloneqq \psi(\xi-1)\ind_{[1,\infty)} (\xi), \qquad \xi \in \mathbb{R}.
\]
One then has
\[
\widecheck{m_{(k,l)}} ( x ) = \frac{e^{i 2 \pi 2^k x}}{i 2 \pi x} + O \left( \frac{1}{2^l |x|} \right), \qquad |x| \gtrsim 2^{-l}. 
\]
For $N \in \mathbb{N}$, that will be eventually sent to infinity, we consider the second order $\ell^2$-valued multiplier operator
\[
 T_N (g)  \coloneq    \left\{ T_{m_{(k,l)}} (g) \right\}_{ 1 \leq l < k \leq N } . 
\]
Consider a smooth function $f$ such that  $\widehat{f}$ is supported in $ [-4, 4]$ and $\widehat{f} (\xi) = 1$ for all $\xi \in [-2, 2]$. We then set $f_N (x) \coloneq 2^N f (2^N x)$, $x \in \mathbb{R}$. For $r>0$ we have
\[
\| T_N (f_N) (x) \|_{\ell^2}  \gtrsim \frac{N}{|x|} \qquad \text{if}\quad |x| \geq 2^{-5N/8} 
\]
and
\[
\int_\mathbb{R} \frac{|f_N (x)|}{\alpha} \left(\log \left(e+\frac{|f_N (x)|}{\alpha}\right)\right)^r \, \d x \lesssim \frac{1}{\alpha} \left(\log  \left(e+\frac{2^N}{\alpha} \right) \right)^r\qquad \text{for}\quad \alpha>0.
\]
Hence, if we choose $\alpha = 2^{5N/8}$ then $ \alpha^{-1} \left(\log (e+\alpha^{-1} 2^N )\right)^r \simeq 2^{-5N/8} N^r$ and
\[
\begin{split}
 \left|\left \{ x \in [-1/2,1/2] : \, \left \| T_N (f_N) (x) \right\|_{\ell^2}  > \alpha \right\} \right| &\geq  \left| \left\{   2^{-5N/8}  \leq x \leq 1/4 :\, \frac{ N }{|x|} \gtrsim 2^{5N/8} \right\} \right| 
 \\
 &\simeq N \alpha^{-1}  .
 \end{split}
 \]
To complete the proof, define $g_N \coloneqq f_N \chi_{[1/2,1/2]}$ so that  $g_N$ is supported in $[-1/2,1/2] $ and $\| g_N \|_{L \log^r L ([-1/2,1/2])} \lesssim N^r$. Moreover, for all $1 \leq l < k \leq N$ one has
\[
 | T_{m_{k,l}} (f_N -g_N) (x) | \lesssim 2^{-2N} \qquad \text{for all}\quad    x \in [ 2^{-5N/8}  , 1/4]
 \]
and hence
\[
\left\|\left \| T_N (g_N)\right \|_{\ell^2} \right\|_{L^{1,\infty} ([-1/2,1/2])} \gtrsim N. 
\]
It follows from Khintchine's inequality that there exists a choice of signs $\eps_{k,\ell}$, depending on $g_N$, such that
\[
\left\| \sum_{1\leq l <k \leq N} \eps_{k,\ell} T_{m_{k,l}} (g_N)  \right\|_{L^{1,\infty} ([-1/2,1/2])} \gtrsim N 
\]
from which it follows that $r \geq 1=\tau/2$.

\subsection{Proof of Theorems~\ref{cor:LPsf} and~\ref{thm:horm}}\label{sec:cormult} We begin by explaining the modifications needed in order to obtain a proof of the endpoint bounds in Theorems~\ref{cor:LPsf} and~\ref{thm:horm}. 

\subsubsection{Proof of Theorem~\ref{thm:Marcink}} Since Marcinkiewicz multipliers of order $\tau$ are contained in the class $R_{2,\tau}$, we only need to briefly discuss the conclusion of Theorem~\ref{thm:Marcink} for the Littlewood--Paley square function. Note that the proof of Theorem~\ref{thm:main} relies on $L^2(\R)$ estimates and $L^1$-type estimates. Then we can repeat the proof for the operator
\[
|\mathrm {LP}_{\tau}f| \simeq\mathbf E\left|\sum_{L\in\Lambda_\tau} \eps_{L} \mathrm P_{L}f\right|\simeq \left(\mathbf E\left|\sum_{L\in\Lambda_\tau} \eps_{L} \mathrm P_{L}f\right|^2\right)^{1/2}
\]
using the first approximate equality whenever $L^1$-estimates are needed, and the second one for the $L^2$-estimates. We omit the  details. The optimality follows by the discussion in \S\ref{sec:optr2tau}. 

\subsubsection{Proof of Theorem~\ref{thm:horm}} We proceed to prove Theorem~\ref{thm:horm} concerning endpoint bounds for higher order H\"ormander-Mihlin multipliers and smooth Littlewood--Paley square functions, which requires just small modifications compared to the proof of Theorem~\ref{thm:main}. Consider a positive integer $\tau$ and $f\in L^{B_{\tau-1}}$; we apply the Calder\'on-Zygmund decomposition of Proposition~\ref{prop:CZ_decomp} with $\sigma=\tau-1$ at some fixed level $\alpha>0$ to write $f=g+b_{\mathrm{canc},\tau-1}+b_{\mathrm{lac},\tau-1}$ and let $\mathcal J$ be the collection of stopping intervals. The good part $g+b_{\mathrm{lac},\tau-1}$ is estimated in $L^2$ by the $L^2$-bounds of the operator $\mathrm T_m$, using that 
\[
 \left\|g+b_{\mathrm{lac},\tau-1}\right\|_{L^2(\R)} ^2 \lesssim \alpha^2 \int_\R B_{(\tau-1)/2}\left(\frac{|f|}{\alpha}\right)
\]
by the Calder\'on-Zygmund decomposition. As before, it remains to estimate the part of the operator acting on the cancellative atoms. We consider a partition of unity $\{\widetilde{\phi}_L\}_{L\in\Lambda_{\tau-1}}$ subordinated to the collection of Littlewood--Paley intervals $\Lambda_{\tau-1}$, with $\widetilde{\phi}_{L}\in\Phi_{L,M}$ for each $L$. We set
\[
\widetilde{\Delta}_L (g)\coloneqq \left(\widetilde{\phi}_{L} \hat g\right)^\vee,\qquad \mathrm{Id}=\sum_{L\in\Lambda _\sigma}\widetilde{\Delta}_L .
\]  
With $\Delta_{L}$ the smooth Littlewood--Paley projections as fixed in \S\ref{sec:lacnot} we have $\Delta_{L}\widetilde{\Delta}_L=\widetilde{\Delta}_L$. We have thus the decomposition
\[
\mathrm{T}_m=\sum_{L\in\Lambda_{\tau-1}} \mathrm{T}_m \widetilde{\Delta}_L\eqqcolon \sum_{L\in\Lambda_{\tau-1}} \mathrm{T}_L=\sum_{L\in\Lambda_{\tau-1}} \mathrm{T}_L \Delta_L
\]
and let $\zeta_{L}$ denote the Fourier multiplier of the operator $\mathrm{T}_{L}$. We then estimate
\[
\begin{split}
\mathrm{T}_m\left(\sum_J b_J\right)&=\sum_{L\in \Lambda_{\tau-1}}\mathrm{T}_{L} \left(\sum_{J:\, |J|\geq  |L|^{-1}}{\Delta_{L}}( b_J)\ind_{\R\setminus 3J}\right)+
\sum_{L\in\Lambda_{\tau-1}} \mathrm{T}_L \left(\sum_{J:\, |J|<  |L|^{-1}}\Delta_L( b_J)\right)
\\
&\qquad +\sum_{L\in\Lambda_{\tau-1}} \mathrm{T}_L\left(\sum_{J:\, |J|\geq  |L|^{-1}} \Delta_L( b_J)\ind_{3J}\right)\eqqcolon \mathrm{I}+\mathrm{II}+\mathrm{III}.
\end{split}
\]
As in the proof of Theorem~\ref{thm:main}, Remarks~\ref{rmrk:molecule_pos} and~\ref{rmrk:molecule_canc} take care of the terms $\mathrm{I}$ and $\mathrm{II}$, respectively, by using $L^2$-bounds for each $\mathrm T_L$ and $L^2$-orthogonality for smooth Littlewood--Paley projections of order $\tau$. Once again the main term is $\mathrm{III}$.

We will split $\mathrm{III}$ into two parts, which are defined in the same way as the operators $\mathcal E$ and $\mathcal L$ from the proof of Theorem~\ref{thm:main} with the role of the interval $I$ being replaced by an interval $L\in\Lambda_{\tau-1}$. For the first part consider for each $L,J$ the function $\psi_{L,J}$ as defined before the proof of Lemma~\ref{lem:local_sing}. Defining
\[
{\mathcal E}(\{b_J\}_{J\in \mathcal J})\coloneqq \sum_{L\in \Lambda_{\tau-1} } \mathrm T_L \left(\sum_{J:\, |J|\geq  |L|^{-1}} \Psi_{L,J}\left({\Delta_L}( b_J)\ind_{3J}\right)\right),
\]
and following the same steps as in the proof of Lemma~\ref{lem:local_sing}, we get
\[
\left|\left\{\left|{\mathcal E}(\{b_J\}_{J\in \mathcal J})\right|>\alpha \right\}\right|\lesssim \frac{1}{\alpha^2}\sum_{J\in\mathcal J}|J|\sum_{L\in\Lambda_{\tau-1}^{|J|^{-1}}}\langle| \Delta_L( b_J)|\rangle_{1,3J}^2\lesssim  \int_{\R}B_{(\tau-1)/2}\left(\frac{|f|}{\lambda}\right)
\]
by the generalized Zygmund--Bonami inequality of Proposition~\ref{prop:genzygpos} and the properties of the Calder\'on-Zygmund decomposition. It remains to deal with the operator
\[
{\mathcal L} (\{b_J\}_{J\in\mathcal J}) \coloneqq \sum_{ L\in\Lambda_{\tau-1}}   \sum_{|J|\geq |L|^{-1}}\mathrm{T}_{L}\left(\mathrm{Id}-\Psi_{ L,J}\right)\left(\Delta_{L}(b_J)\ind_{3J}\right).
\]
Letting $\zeta_{L,J}$ be the Fourier multiplier of the operator $\mathrm{T}_{L}(\mathrm{Id}-\Psi_{L,J})$ and  $\rho_J$ as in the statement of Lemma~\ref{lem:ngtmr} we set $\widecheck{\mu_{L,J}} \coloneqq (1-\rho_J)  \widecheck{\zeta_{L,J}}$. Lemma~\ref{lem:ngtmr} for $I=L\in\Lambda_{\tau-1}$ yields the estimate
\begin{equation}\label{eq:derivatives}
\left|\partial_\xi  ^\beta  \mu_{L,J}(\xi)\right|\lesssim |J|^{\beta}  \left(1+|J|\dist(\xi,\R
\setminus L)\right)^{-\gamma} , \qquad \xi\in\R.
\end{equation}
The proof for the operator $\mathcal L$ is then completed in the by now usual way. First we have
\[
\begin{split}
 \left|\left\{\left|{\mathcal L} (\{b_J\}_{J\in\mathcal J}) \right|>\alpha\right\}\right| \leq\frac{1}{\alpha} \sum_{J\in\mathcal J}|J|^{-1/2} \left\| \partial_\xi\left(\sum_{L\in\Lambda_{\tau-1} ^{|J|^{-1}}} \mu_{L,J}\widehat{F_{L,J}}\right)\right\|_{L^2(\R)} 
\end{split}
\]
with $F_{L,J}\coloneqq \widetilde{\Delta}_{L}(b_J)\ind_{3J}$. Now \eqref{eq:derivatives} implies that
\[
\left| \partial_\xi\left(\sum_{L\in\Lambda_\sigma ^{|J|^{-1}}} \mu_{L,J}\widehat{F_{L,J}}\right)\right|\lesssim |J|^2\M(\ind_{L_J})^\gamma \langle|\widetilde{\Delta}_{L}(b_J)|\rangle_{1,3J}
\]
where $L_{ J}\coloneqq [\xi_L +(10|J|)^{-1},\xi_L +(5|J|)^{-1}]\subset L$. The estimates above together with the Fefferman-Stein inequality, the Cauchy-Schwarz  inequality and the generalized Zygmund--Bonami inequality complete the estimate for the operator $\mathcal{L}$ and with that the upper bound of Theorem~\ref{thm:horm} for H\"ormander-Mihlin multipliers of order $\tau$. The proof for the smooth Littlewood--Paley square function of order $\tau$ follows the same randomization argument as the one used in the proof of Theorem~\ref{thm:Marcink}.

Finally, the optimality of the power $(\tau-1)/2$ on the endpoint inequality can be checked by testing a local endpoint $L\log^{r}L\to L^{1,\infty}$ inequality for the smooth Littlewood--Paley square function of order $\tau$ on a smooth bump function supported in a small neighborhood of the origin. A routine calculation shows that necessarily $r\geq (\tau -1 )/2$. Note also that a local $L\log^{r}L\to L^{1,\infty}$ bound for H\"ormander-Mihlin multipliers implies the corresponding endpoint square function estimate by a randomization argument as in \cite{Bakas}.

\subsection*{Acknowledgements} 

O. Bakas would like to thank Jim Wright for several discussions concerning topics related to this work. V. Ciccone and M. Vitturi thank BCAM - Basque Center for Applied Mathematics and the
University of the Basque Country UPV/EHU for their warm hospitality during their respective research visits. The authors would like to thank the anonymous referees for an expert reading and suggestions to improve the paper.

\end{document}